\documentclass[12pt]{article}
\usepackage{amscd, amssymb, amsthm, amsmath}
\newtheorem{thm}{Theorem}
\newtheorem{defn}{Definition}
\newtheorem{lem}{Lemma}
\newtheorem{cor}{Corollary}
\newtheorem{pro}{Proposition}
\newtheorem{rem}{Remark}
\author{Yng-Ing Lee* and Yang-Kai Lue**}
\date{}
\title{The Stability of Self-Shrinkers of Mean Curvature Flow in Higher Codimension}
\begin{document}
\maketitle   \leftline{*Department of Mathematics, National Taiwan
University, Taipei, Taiwan} \leftline{\quad National Center for
Theoretical Sciences, Taipei Office, Taipei, Taiwan}
 \centerline{email: yilee@math.ntu.edu.tw} \leftline{**Department of Mathematics, National Taiwan
University, Taipei, Taiwan} \centerline{email: luf961@yahoo.com.tw}
 \begin{abstract}
 In this paper, we generalize Colding and Minicozzi's work \cite{CM} on the stability
 of self-shrinkers in the hypersurface case to higher co-dimensional cases. The first and
second variation formulae of the $F$-functional are derived and an
equivalent condition to the stability in general codimension is
found. Moreover, we show that the closed Lagrangian self-shrinkers
given by Anciaux in \cite{An} are unstable.
  \end{abstract}
  \section{Introduction}
Let $X:\Sigma\to \mathbb R^m$ be an isometric immersion of an
$n$-dimensional manifold $\Sigma$ in the Euclidean space $\mathbb
R^m$. The mean curvature flow of $X$ is a family of immersions
$X_t:\Sigma\to\mathbb R^m$ that satisfies
\begin{align}
\left(\frac{\partial}{\partial t}X_t(x)\right)^{\perp}&=H(x,t) \notag \\
                      X_0&=X  \notag
\end{align}
where $H(x,t)$ is the mean curvature vector of $X_t(\Sigma)$ at
$X_t(x)$ and $v^{\perp}$ denotes the projection of $v$ into the
normal space of $X_t(\Sigma)$. Mean curvature flow of a submanifold
in a Riemannian manifold can also be defined similarly. Because the
mean curvature vector points in the direction in which the area
decreases most rapidly, mean curvature flow is thus a
 canonical way to construct minimal submanifolds. It also improves
 the geometric properties of a object along the flow (e.g., see \cite{Hu1}) 

A submanifold $\Sigma$ in $\mathbb R^m$ is called a \emph{self-shrinker} if its
position vector $X:\Sigma\to \mathbb R^m$ satisfies
\begin{align}
H=-\frac{1}{2}X^{\perp}. \notag
\end{align}
The terminology comes from the fact that $\sqrt{1-t}X(\Sigma)$ is a solution of
mean curvature flow, i.e., a self-shrinker evolves homothetically along mean
curvature flow in
a shrinking way. Moreover, self-shrinkers describe all possible
central blow-up limits of a finite-time singularity of the mean curvature
flow. This follows from Huisken's
monotonicity formula \cite{Hu2}, and its generalization to type II
singularity by Ilmanen \cite{I1} and White.
 Singularities will occur in general along mean curvature flow and are
 obstacles to continue the flow.  It is therefore an important issue to understand
 singularities and the candidates of their blow-up limits, self-shrinkers.

 Standard sphere $\mathbb S^n(\sqrt{2n})$ and
cylinder $\mathbb S^{k}(\sqrt{2k})\times \mathbb R^{n-k}$ are simple examples
of self-shrinkers in $\mathbb R^m$.  Abresch and
Langer \cite{AL} found all immersed closed self-shrinkers in
the plane. For other complete hypersurfaces case,
 Huisken \cite{Hu3} classified all self-shrinkers with nonnegative
mean curvature and bounded geometry. The bounded geometry condition
is later weakened to polynomial volume growth by Colding and Minicozzi  in \cite{CM}.
On the other hand,
many other different co-dimension one self-shrinkers are found (e.g., see \cite{A}),
and a classification of all self-shrinkers is not expected. Our understanding on
self-shrinkers
in higher co-dimension is even more limited.  Smoczyk obtained a classification
for self-shrinkers with parallel
principal normal $\nu\equiv H/|H|$  and bounded geometry in \cite{Sm}.
Various different families of
Lagrangian self-shrinkers,  which are of middle dimension,
 are constructed in \cite{An}, \cite{LW2} and \cite{JLT}.

 Adapted from the back heat kernel introduced by Huisken in \cite{Hu2}, Colding and Minicozzi
 \cite{CM} defined a functional $F$  by
\begin{equation}
    F(\Sigma,x,t)=\frac{1}{\sqrt{4 \pi t}^n}\int_{\Sigma} e^{\frac{-|X-x|^2}{4t}}d\mu,\label{I1}
\end{equation}
for any  submanifold  $X:\Sigma^n\to\mathbb R^{n+1}$, $x\in \mathbb
R^{n+1}$ and $t>0$. One of the main properties of this functional is
that $(\Sigma, x_0,t_0)$ is a critical point of $F$ iff $\Sigma$
satisfies $H=-\frac{(X-x_0)^{\perp}}{2t_0}$. Especially, it is a
self-shrinker when $x_0=0$ and $t_0=1$. They proved that if an
$n$-dimensional complete smooth embedded self-shrinker $\Sigma^n$
without boundary and with polynomial volume growth in $\mathbb
R^{n+1}$ is $F$-stable with respect to compactly supported
variations, then it must be the round sphere or a hyperplane. Here
$F$-stable means that for every compactly supported smooth variation
$\Sigma_s$ with $\Sigma_0=\Sigma$, there exist variations $x_s$ of
$0$ and $t_s$ of $1$ such that $\frac{\partial^2}{\partial
s^2}F(\Sigma_s, x_s, t_s)\geq0$ at $s=0$. The importance of the
study is that  roughly speaking,
 the blow-up near type-I singularity of mean curvature flow for generic
initial data
gives stable self-shrinkers (see \cite{CM} for exact statement.)


In this paper, we intend to generalize Colding and Minicozzi's work
\cite{CM} to higher co-dimensional cases. The domain of the
functional $F$ is now $(\Sigma, x, t)$ for $\Sigma^n\subset\mathbb
R^m$, $x\in\mathbb R^m$ and $t>0$. Colding and Minicozzi's
classification on stable self-shrinkers in co-dimension one is first
to conclude that the mean curvature function $h$ is the first
eigenvalue of an elliptic operator, it then implies $h\ge  0$, and
Huisken's classification of self-shrinkers with nonnegative $h$ will
lead to the conclusion. Although the counter part of Huisken's
result in higher co-dimension is  still not available, we can also
pin down the stability of self-shrinkers in higher co-dimension to
the mean curvature vector being the first vector-valued
eigenfunction for an elliptic system. More precisely, the equivalent
condition of stabilities is as in the following Theorems.

\noindent {\bf Theorem \ref{thm3} } {\it Suppose $\Sigma \subset
\mathbb{R}^m$ is an $n$-dimensional smooth closed self-shrinker
$H=-\frac{X^{\perp}}{2}$. The following statements are equivalent:

(i) $\Sigma$ is F-stable.

(ii) $\int_{\Sigma}\langle V,-L^{\perp}V\rangle
e^{-\frac{|X|^2}{4}}d\mu\geq0$ for any smooth normal vector field
$V$ which satisfies
\begin{align}
\int_{\Sigma}\langle V,H\rangle e^{-\frac{|X|^2}{4}}d\mu=0 \quad\mbox{and}\quad
\int_{\Sigma}Ve^{-\frac{|X|^2}{4}}d\mu=\overrightarrow{0}, \notag
\end{align}
where $L^{\perp}V=\Delta^{\perp}V+\langle A_{ij},V\rangle
g^{ki}g^{jl}A_{kl}+\frac{V}{2}
-\frac{1}{2}\nabla^{\perp}_{X^{\top}}V$ is a second order elliptic operator
and $A_{ij}$ is the second fundamental form as definition in
(\ref{201}), and $\nabla^{\perp}$ is the normal connection of
$\Sigma$. }

For the complete case, we define
\begin{align}
S_{\Sigma}=\{V\in N\Sigma\big|\ \ |V|(X) \mbox{ and } |\nabla^{\perp}V|(X)
\mbox{ are of polynomial growth } \}. \notag
\end{align}
Note that $V\in S_{\Sigma}$ is not necessarily with compact support.
The equivalent condition for the stability of
$F$ in the complete case becomes

\noindent {\bf Theorem \ref{t4} } {\it Let $\Sigma \subset
\mathbb{R}^m$ be an $n$-dimensional smooth complete self-shrinker
$H=-\frac{X^{\perp}}{2}$ without boundary. Suppose that the second
fundamental form $A$ of $\Sigma$ is of polynomial growth and
$\Sigma$ has polynomial volume growth. The following statements are
equivalent:

(i) $\Sigma$ is $F$-stable.

(ii) $\int_{\Sigma}\langle V,-L^{\perp}V\rangle
e^{-\frac{|X|^2}{4}}d\mu\geq0$ for any smooth normal vector field
$V$ in $S_{\Sigma}$ which satisfies
\begin{align}
\int_{\Sigma}\langle V,H\rangle e^{-\frac{|X|^2}{4}}d\mu=0 \quad \mbox{and} \quad
\int_{\Sigma}Ve^{-\frac{|X|^2}{4}}d\mu=\overrightarrow{0}. \notag
\end{align}
}

Using Theorem \ref{thm3} and \ref{t4}, we immediately can conclude that
 the product of any two non-trivial self-shrinkers, which is also a
self-shrinker, is $F$-unstable.

\noindent {\bf Corollary \ref{c1} } {\it Suppose
$\Sigma_i^{n_i}\subset\mathbb R^{m_i}$, $i=1,2$, are smooth closed
self-shrinkers which satisfy $H_i=-\frac{X_i^{\perp}}{2}$, where
$X_i$ are the position vectors of $\Sigma_i$. Then $\Sigma_1\times
\Sigma_2\subset\mathbb R^{m_1+m_2}$ is a self-shrinker and is
F-unstable.}

\noindent {\bf Corollary \ref{c2} } {\it Let
$\Sigma_i^{n_i}\subset\mathbb R^{m_i}$, $i=1,2$, be two smooth
complete self-shrinkers without boundary which satisfy
$H_i=-\frac{X_i^{\perp}}{2}\neq0$, where $X_i$ is the position vector
of $\Sigma_i$. Suppose that each $\Sigma_i$ has polynomial volume
growth and the second fundamental form of $\Sigma_i$ is of
polynomial growth. Then $\Sigma_1\times \Sigma_2\subset\mathbb
R^{m_1+m_2}$ is a self-shrinker and is F-unstable. }

 Note that we in fact allow the self-shrinkers to be immersed in our discussion.
The examples $\mathbb S^n(\sqrt{2n})$ and $\mathbb R^n$ are still
stable self-shrinkers in $\mathbb R^m$, but the situation for all other
higher co-dimensional examples is not clear.
 We employ the above equivalent condition to investigate
the $F$-stability of the Lagrangian self-shrinkers constructed by Anciaux in
\cite{An} in Section $4$. These $n$-dimensional self-shrinkers in
$\mathbb C^n$, $n \geq 2,$ are expressed as $\gamma(s)\psi(\sigma)$, where
$\psi:M^{n-1}\to \mathbb S^{2n-1}\subset \mathbb C^n$ is a minimal Legendrian immersion
and $\gamma$ is a complex-valued function that satisfies the system
of ordinary differential equations
(\ref{l201}). We prove that

\noindent {\bf Theorem \ref{t5} } {\it Anciaux's closed examples
 as described in Lemma~\ref{l2}
is $F$-unstable. }

Since Anciaux's examples are Lagrangian in $\mathbb
C^n$, it is natural to ask whether these examples are
still $F$-unstable under the restricted Lagrangian variations.
We have the following

\noindent {\bf Theorem \ref{thm6} } {\it Anciaux's closed examples is $F$-unstable under
Lagrangian variations for the following cases
\begin{itemize}
\item[\rm{(i)}] $n=2$ or  $n\geq7$,
\item[\rm{(ii)}] $2<n<7$, and
$E\in[\frac{1}{\sqrt{2}}E_{max},E_{max}]$,
\end{itemize}
where $E$ and $E_{max}$ are described in ($\ref{l202}$). }

{\bf Acknowledgements:} The authors are grateful to Mu-Tao Wang for his constant
support and interest in this work. The first author also like to thank Jacob Bernstein's
discussions.
\section{The 1st and 2nd variation formulae of $F$}
\subsection{Notations and preliminaries} Let $X:\Sigma^n \to
\mathbb{R}^{m}$ be a smooth isometric immersion of a submanifold of
codimension $m-n$. If $\{e_i\}$ and $\{e_{\alpha}\}$ are orthonormal
frames for the tangent bundle $T\Sigma$ and the normal bundle
$N\Sigma$, respectively, then the coefficients of the second
fundamental form and the mean curvature vector are defined to be
\begin{align}
A_{ij}&=A^{\alpha}_{ij}e_{\alpha}\equiv\langle\overline{\nabla}_{e_i}e_j,
e_{\alpha}\rangle e_{\alpha} \label{201}\\
\mbox{and}\quad H&=H^{\alpha}e_{\alpha}\equiv A_{ii}, \notag
\end{align}
where by convention we are summing over repeated indices and
$\overline{\nabla}$ is the standard connection of the ambient
Euclidean space. For a submanifold $B$ in an ambient manifold  $C$,
we use $A^{B,C}$ and $H^{B,C}$ to denote the associated second
fundamental form and mean curvature vector, respectively. When the
ambient space is $\mathbb{C}^n$, we denote them as $A^B$ (or $A$)
and $H^B$ (or $H$) for simplicity. Given a normal vector field $V$
in $N\Sigma$, $\langle A,V\rangle$ is a $(2,0)-$tensor and $|\langle
A,V\rangle|^2$ is defined as
$\underset{i,j=1}{\overset{n}{\sum}}\langle A_{ij},V\rangle^2$. When
$\Sigma$ is a hypersurface, the mean curvature vector $H$ and the
second fundamental form reduce to the function $h=-\langle
H,\textbf{n}\rangle$ and the 2-tensor $h_{ij}=-\langle
A_{ij},\textbf{n}\rangle$, respectively. Here $\textbf{n}$ is the
unit outer normal vector of $\Sigma$.
\begin{defn}
{\rm Let $\Sigma$ be a submanifold in $\mathbb R^m$ and $B_r(0)$ be
the geodesic ball in $\mathbb R^m$ with radius $r$. $\Sigma$ is said
to have} {\it polynomial volume growth} {\rm if there are constants
$C_1$, $C_2$ and $k\in\mathbb N$ so that for all $r\geq0$}
\begin{align}
Vol(B_r(0)\cap\Sigma)\leq C_1r^k+C_2. \notag
\end{align}
\end{defn}
\begin{defn} \label{d1}
{\rm A normal vector field $V$ (or the second fundamental form $A$)
of $\Sigma$ is of} {\it polynomial growth} {\rm if there are
constants $C_1$, $C_2$ and $k\in\mathbb N$ so that for all $r\geq0$}
\begin{align}
|V|\leq C_1r^k+C_2 \quad (\mbox{\rm or}\quad
|A|\leq C_1r^k+C_2) \quad \mbox{\rm on} \quad B_r(0)\cap\Sigma.\notag
\end{align}
\end{defn}
The space of all normal vector fields with polynomial growth is
denoted by $P\Gamma(N\Sigma)$. For any two normal vector fields $V$
and $W$ in $P\Gamma(N\Sigma)$, its weighted inner product, denoted
as $\langle V,W\rangle_e$, is defined to be $\int_{\Sigma}\langle
V,W\rangle e^{-\frac{|X|^2}{4}}d\mu$. The space $(P\Gamma(N\Sigma),
\langle\cdot,\cdot\rangle_e)$ is called the {\it weighted inner
product space.}

\subsection{The first variation formula of $F$}
Colding and Minicozzi derived the first and second variation
formulae of the $F$ functionals of a hypersurface in \cite{CM}.
These can be generalized to higher co-dimensional cases by similar
calculation.
We derive the first variation formula of $F$ in the following
Theorem.
\begin{thm}
Let $\Sigma\subset\mathbb R^m$ be an $n$-dimentional complete
manifold without boundary which has polynomial volume growth.
Suppose that $\Sigma_{s}\subset\mathbb R^m$ is a normal variation of
$\Sigma$, $x_{s}$, $t_{s}$ are variations of $x_0$ and $t_0$, and
\begin{equation}
\frac{\partial\Sigma_s}{\partial s}=V,\quad\frac{\partial x_s}{\partial s}=y,
 \quad and \quad \frac{\partial t_s}{\partial s}=\tau, \notag
\end{equation}
where $V$ has compact support. Then
\begin{align}\label{l101}
\frac{\partial}{\partial s}F(\Sigma_s,x_s,t_s)=
\frac{1}{\sqrt{4\pi t_s}^n}\int_{\Sigma_s}\Bigl(&
-\langle V,H_s+\frac{X_s-x_s}{2t_s}\rangle
+\tau(\frac{|X_s-x_s|^2}{4t_s^2}-\frac{n}{2t_s}) \notag \\
&+\frac{\langle X_s-x_s,y\rangle}{2t_s}\Bigr)e^{\frac{-|X_s-x_s|^2}{4t_s}}d\mu,
\end{align}
where $X_s$ is the position vector of $\Sigma_s$ and $H_s$ is its
mean curvature vector.
\end{thm}
\begin{proof}
From the first variation formula for area, we know that
\begin{equation}\label{l11}
\frac{\partial}{\partial s}(d\mu)=-\langle H_s,V\rangle d\mu.
\end{equation}
The variation of the weight $\frac{1}{\sqrt{4\pi
t_s}^n}e^{-|X_s-x_s|^2/4t_s}$ have terms coming from the variation
of $X_s$, the variation of $x_s$ and the variation of $t_s$,
respectively. Using the following equations
\begin{align}
\frac{\partial}{\partial {t_s}} \log\bigl((4\pi t_s)^{-n/2}e^{-\frac{|X_s-x_s|^2}{4t_s}}\bigr)
&=\frac{-n}{2t_s}+\frac{|X_s-x_s|^2}{4t_s^2}, \notag \\
\frac{\partial}{\partial {x_s}} \log\bigl((4\pi t_s)^{-n/2}e^{-\frac{|X_s-x_s|^2}{4t_s}}\bigr)
&=\frac{X_s-x_s}{2t_s} \notag \\
\mbox{and  \ \ \ }
\frac{\partial}{\partial {X_s}} \log\bigl((4\pi t_s)^{-n/2}e^{-\frac{|X_s-x_s|^2}{4t_s}}\bigr)
&=-\frac{X_s-x_s}{2t_s}, \notag
\end{align}
we obtain
\begin{align}
&\frac{\partial}{\partial s}\log\bigl((4\pi
t_s)^{-n/2}e^{-\frac{|X_s-x_s|^2}{4t_s}}\bigr) \notag \\
=&-\frac{\langle X_s-x_s,V\rangle}{2t_s}+\tau(\frac{|X_s-x_s|^2}{4t_s^2}-\frac{n}{2t_s})
+\frac{1}{2t_s}\langle X_s-x_s,y\rangle. \notag
\end{align}
Combining this with (\ref{l11}) gives (\ref{l101}).
\end{proof}
\begin{defn}
{\rm We will call $(\Sigma, x_0, t_0)$ a} {\it critical point} {\rm
of $F$ if it is critical with respect to all normal variations which
have compact support in $\Sigma$ and all variations in $x$ and $t$.}
\end{defn}
From the definition of $F$ in (\ref{I1}), we have
$F(\Sigma,x,t)=F(\frac{\Sigma-x}{\sqrt{t}},0,1)$ and it is easy to
see the following property:
\begin{align}
&(\Sigma,x_0,t_0) \mbox{ is a critical point of } F
\mbox{ if and only if } (\frac{\Sigma-x_0}{\sqrt{t_0}},0,1) \notag \\
&\mbox{ is a critical point of } F. \label{p001}
\end{align}
Therefore, we only consider the case $x_0=0$, $t{_0}=1$. In the case
of hypersurfaces, Colding and Minicozzi proved that, $(\Sigma,0,1)$
is a critical point of $F$ if $\Sigma$ satisfies that
$h=\frac{\langle X,\textbf{n}\rangle}{2}$. Their result, when
written in the vector form $H=-\frac{X^{\perp}}{2}$, also holds for
higher co-dimensional cases. The proof needs following propositions.
\begin{pro} \label{p1}
If $\Sigma \subset \mathbb{R}^m$ is an $n$-dimensional complete
submanifold with $H=-\frac{X^{\perp}}{2}$, then
\begin{align}
\mathcal{L} X_i&=-\frac{1}{2}X_i \quad\mbox{and}\notag \\
\mathcal{L} |X|^2&=2n-|X|^2. \label{p102}
\end{align}
Here $X_i$ is the i-th component of the position vector $X$, i.e.,
$X_i=\langle X,\partial_i\rangle$ and the linear operator
$\mathcal{L} f=\Delta f-\frac{1}{2}\langle X,\nabla
f\rangle=e^{\frac{|X|^2}{4}}div(e^{\frac{-|X|^2}{4}}\nabla f)$.
\end{pro}
\begin{pro} \label{p2}
If $\Sigma \subset \mathbb{R}^m$ is an $n$-dimensional complete
submanifold, $\partial \Sigma =\emptyset$, $\Sigma$ has polynomial
volume growth, and $H=-\frac{X^{\perp}}{2}$, then
\begin{align}
&\int_{\Sigma}Xe^{\frac{-|X|^2}{4}}d\mu=\overrightarrow{0}
=\int_{\Sigma}X|X|^2e^{\frac{-|X|^2}{4}}d\mu \quad\mbox{and} \notag\\
&\int_{\Sigma}(|X|^2-2n)e^{\frac{-|X|^2}{4}}d\mu=0.  \label{p201}
\end{align}
Moreover, if $W\in\mathbb R^m$ is a constant vector, then
\begin{align}
\int_{\Sigma}\langle X,W\rangle^2e^{-\frac{|X|^2}{4}}d\mu
=2\int_{\Sigma}|W^{\top}|^2e^{-\frac{|X|^2}{4}}d\mu. \label{p202}
\end{align}
\end{pro}
These propositions were proved by Colding and Minicozzi in the case
of hypersurfaces (see Lemma 3.20 and Lemma 3.25 in \cite{CM}). We
omit the proofs here because the argument is similar. Combining
(\ref{l101}), (\ref{p001}) and (\ref{p201}), we get
\begin{pro}
For any $x_0\in \mathbb R^m$, $t_0\in\mathbb R^+$,
$(\Sigma,x_0,t_0)$ is a critical point of $F$ if and only if
$H=-\frac{(X-x_0)^{\perp}}{2t_0}$.
\end{pro}
\subsection{The general second variation formula of $F$}
\begin{thm} \label{thm1}
Let $\Sigma$ be an $n$-dimensional complete manifold without
boundary which has polynomial volume growth. Suppose that
$\Sigma_{s}$ is a normal variation of $\Sigma$, $x_{s}$, $t_{s}$ are
variations of $x_0$ and $t_0$, and
\begin{equation}
\frac{\partial\Sigma_s}{\partial s}=V,\quad \frac{\partial x_s}{\partial s}=y,
\quad \frac{\partial t_s}{\partial s}=\tau, \quad
\frac{\partial^2x_s}{\partial s^2}=y',
\quad and \quad\frac{\partial^2t_s}{\partial s^2}=\tau', \notag
\end{equation}
where $V$ has compact support. Then
\begin{align}
&\frac{\partial^2F}{\partial s^2}(\Sigma,x_0,t_0) \notag \\
=&\frac{1}{\sqrt{4\pi t_0}^n}\int_{\Sigma}e^{-\frac{|X-x_0|^2}{4t_0}}\Big\{
-\langle V,L^{\perp}_{x_0,t_0}V\rangle+\frac{\langle X-x_0,V\rangle}{t_0^2}\tau+\frac{\langle V,y\rangle}{t_0}
\notag \\
&-\frac{(|X-x_0|^2-nt_0)\tau^2}{2t_0^3}-\frac{|y|^2}{2t_0}-\frac{\tau\langle X-x_0,y\rangle}{t_0^2}\notag \\
&+\left(-\langle V,H+\frac{X-x_0}{2t_0}\rangle+\tau(\frac{|X-x_0|^2}{4t_0^2}-\frac{n}{2t_0})+
\langle\frac{X-x_0}{2t_0},y\rangle\right)^2\notag \\
&-\langle\overline{\nabla}_V^{\perp}V,H+\frac{X-x_0}{2t_0}\rangle+\tau'(\frac{|X-x_0|^2}{4t_0}-\frac{n}{2t_0})
+\frac{\langle X-x_0,y'\rangle}{2t_0} \Big\} d\mu, \label{t101}
\end{align}
where $L_{x_0,t_0}^{\perp}V=\Delta^{\perp}V+\langle A_{ij},V\rangle
g^{ki}g^{jl}A_{kl}+\frac{V}{2t_0}
-\frac{1}{2t_0}\nabla^{\perp}_{(X-x_0)^{\top}}V$ and $A_{ij}$ is the
second fundamental form as definition in (\ref{201}).
\end{thm}
\begin{proof}
Apply one more derivative on equation (\ref{l101}), it gives
\begin{align}
&\frac{\partial^2F}{\partial s^2}(\Sigma,x_0,t_0) \notag \\
=&\frac{1}{\sqrt{4\pi t_0}^n}\int_{\Sigma}e^{-\frac{|X-x_0|}{4t_0}}
\Bigl\{-\langle V,\frac{\partial}{\partial s}(H_s+\frac{X_s-x_s}{2t_s})\Big|_{s=0}\rangle \notag\\
&+\tau\frac{\partial}{\partial s}(\frac{|X_s-x_s|^2}{4t_s^2}-\frac{n}{2t_s})\Big|_{s=0}
+\langle\frac{\partial}{\partial s}(\frac{X_s-x_s}{2t_s})\Big|_{s=0},y\rangle\notag \\
&+\left(-\langle V,H+\frac{X-x_0}{2t_0}\rangle+\tau(\frac{|X-x_0|^2}{4t_0^2}-\frac{n}{2t_0})+
\langle(\frac{X-x_0}{2t_0}),y\rangle\right)^2\notag \\
&-\langle V',(H+\frac{X-x_0}{2t_0})\rangle
+\tau'(\frac{|X-x_0|^2}{4t_0^2}-\frac{n}{2t_0})+
\langle(\frac{X-x_0}{2t_0}),y'\rangle\Bigr\}d\mu. \label{t11}
\end{align}
Similar to the derivation of the second variation formula for the
area, we have
\begin{align}
\langle(\frac{\partial H_s}{\partial s}), V\rangle
=\langle\Delta^{\perp}V+\langle
A_{ij},V\rangle g^{ki}g^{jl}A_{kl}, V\rangle. \label{t12}
\end{align}
On the other hand, since $[V,(\frac{X-x_0)}{2t_0})^{\top}]$ is
tangent to $\Sigma_s$, it follows that
\begin{align}
\langle\overline{\nabla}^{\top}_VV,\frac{X-x_0}{2t_0}\rangle
=-\langle V,\overline{\nabla}_V(\frac{X-x_0}{2t_0})^{\top}\rangle
=-\langle V,\overline{\nabla}_{(\frac{X-x_0}{2t_0})^{\top}}V\rangle.  \label{t13}
\end{align}
Using $\frac{\partial X_s}{\partial s}=V$, $\frac{\partial
t_s^{-1}}{\partial s}=-\tau t_s^{-2}$ and $\frac{\partial
x_s}{\partial s}=y$, we simplify
\begin{align}
&-\langle V,\frac{\partial}{\partial s}(H_s+\frac{X_s-x_s}{2t_s})\Big|_{s=0}\rangle
-\langle V',H+\frac{X-x_0}{2t_0}\rangle \notag \\
=&-\langle V,\frac{\partial H_s}{\partial s}\Big|_{s=0}\rangle
-\langle V,\frac{\partial}{\partial s}(\frac{X_s-x_s}{2t_s})\Big|_{s=0}\rangle
-\langle\overline{\nabla}^{\perp}_VV,H+\frac{X-x_0}{2t_0}\rangle
-\langle\overline{\nabla}^{\top}_VV,\frac{X-x_0}{2t_0}\rangle \notag \\
=&-\langle V,L^{\perp}_{x_0,t_0}V\rangle
-\langle\overline{\nabla}^{\perp}_VV,H+\frac{X-x_0}{2t_0}\rangle
+\langle V,\frac{y}{2t_0}\rangle+\frac{\tau}{2t_0^2}\langle V,X-x_0\rangle, \notag
\end{align}
where the second equality is from (\ref{t12}), (\ref{t13}), and the
definition of $L^{\perp}_{x_0,t_0}$. The second term in (\ref{t11})
is given by
\begin{align}
\frac{\partial}{\partial s}(\frac{|X_s-x_s|^2}{4t_s^2}-\frac{n}{2t_s})\Big|_{s=0}
&=\frac{\langle X-x_0,V-y\rangle}{2t_0^2}-\frac{\tau|X-x_0|^2}{2t_0^3}+\frac{n\tau}{2t_0^2} \notag \\
&=\frac{\langle X-x_0,V\rangle}{2t_0^2}-\frac{|X-x_0|^2-nt_0}{2t_0^3}
\tau-\frac{\langle X-x_0,y\rangle}{2t_0^2}. \notag
\end{align}
For the third term in (\ref{t11}), observe that
\begin{equation}
\langle\frac{\partial}{\partial s}(\frac{X_s-x_s}{2t_s})\Big|_{s=0},y\rangle
=\langle\frac{V}{2t_0},y\rangle-\frac{|y|^2}{2t_0}
-\frac{\tau}{2t_0^2}\langle X-x_0,y\rangle. \notag
\end{equation}
Combining these gives the theorem.
\end{proof}
\subsection{The second variation formula at a critical point}
For convenience, from now on we denote $D^2_{(V,y,\tau)}F$ as
$\frac{\partial^2F}{\partial s^2}(\Sigma,0,1)$ in (\ref{t101}). When
$(\Sigma,0,1)$ is a critical point of $F$, we have
$H=-\frac{X^{\perp}}{2}$, the second variation formula of $F$ at the
point can be simplified as the following equation (\ref{thm2}).
\begin{thm}
Let $\Sigma$ be a complete manifold without boundary which has
polynomial volume growth. Suppose that  $\Sigma_{s}$ is a normal
variation of $\Sigma$, $x_{s}$, $t_{s}$ are variations of $x_0=0$
and $t_0=1$, and
\begin{equation}
\frac{\partial \Sigma_{s}}{\partial s}\Big|_{s=0}=V,\quad
\frac{\partial x_{s}}{\partial s}\Big|_{s=0}=y,
\quad \frac{\partial t_{s}}{\partial s}\Big|_{s=0}=\tau,\notag
\end{equation}
where $V$ has compact support. If $(\Sigma,0,1)$ is a critical point
of $F$, then
\begin{align}
    &D^2_{(V,y,\tau)}F \notag\\
=&\frac{1}{\sqrt{4\pi}^{n}}
   \int_{\Sigma}\Bigl(-\langle V,L^{\perp}V\rangle
    -2\tau\langle H,V\rangle
    -\tau^2|H|^2+\langle V,y\rangle
    -\frac{1}{2}|y^{\perp}|^2\Bigr)
    e^{-\frac{|X|^2}{4}}d\mu. \label{thm2}
\end{align}
Here the operator $L^{\perp}=L_{0,1}^{\perp}$, and
\begin{equation}\label{L}
L^{\perp}V=\Delta^{\perp}V+\langle A_{ij},V\rangle g^{ki}g^{jl}A_{kl}+\frac{V}{2}
-\frac{1}{2}\nabla^{\perp}_{X^{\top}}V.
\end{equation}
\end{thm}
\begin{proof}
Since $(\Sigma,0,1)$ is a critical point of $F$, by (\ref{l101}) we
have  that
\begin{equation}
H=-\frac{X^{\perp}}{2}. \label{t21}
\end{equation}
It follows from (\ref{p201}) that
\begin{equation}
\int_{\Sigma}Xe^{\frac{-|X|^2}{4}}=\overrightarrow{0}=\int_{\Sigma}X|X|^2e^{\frac{-|X|^2}{4}}
\quad \mbox{and} \quad
\int_{\Sigma}(|X|^2-2n)e^{\frac{-|X|^2}{4}}=0. \label{t22}
\end{equation}
Theorem \ref{thm1} (with $x_0=0$ and $t_0=1$) gives
\begin{align}
D^2_{(V,y,\tau)}F
=(4\pi)^{-\frac{n}{2}}\int_{\Sigma} \Bigl(&
-\langle V,L^{\perp}V\rangle+\tau\langle X,V\rangle+\langle V,y\rangle-\frac{(|X|^2-n)\tau^2}{2}
-\frac{|y|^2}{2} \notag \\
&-\tau\langle X,y\rangle+\{\tau(\frac{|X|^2}{4}-\frac{n}{2})+
\langle\frac{X}{2},y\rangle\}^2\Bigr)e^{-\frac{|X|^2}{4}}d\mu, \notag
\end{align}
where we use (\ref{t21}) and (\ref{t22}) to conclude the vanishing
of a few terms in (\ref{t101}). Note that $y$ is a constant vector
and $\tau$ is a constant. Squaring out the last term of
$D^2_{(V,y,\tau)}F$ and using (\ref{t21}) and (\ref{t22}) again
leads to
\begin{align}
D^2_{(V,y,\tau)}F
=&(4\pi)^{-\frac{n}{2}}\int_{\Sigma}\Bigl(-\langle
V,L^{\perp}V\rangle -2\tau\langle H,V\rangle+\langle V,y\rangle
-\frac{|y|^2}{2}\notag \\
&+\tau^2(\frac{|X|^2}{4}-\frac{n}{2})^2+\frac{1}{4}\langle X,y\rangle^2-\frac{(|X|^2-n)\tau^2}{2}\Bigr)
e^{-\frac{|X|^2}{4}}d\mu. \notag
\end{align}
Using the equality (\ref{p102}) and Stokes' theorem, we have that
\begin{align}
\int_{\Sigma}\tau^2(\frac{|X|^2}{4}-\frac{n}{2})^2e^{-\frac{|X|^2}{4}}d\mu
=\int_{\Sigma}\tau^2\frac{|X^{\top}|^2}{4}e^{-\frac{|X|^2}{4}}d\mu. \notag
\end{align}
Combining (\ref{p201}) and (\ref{p202}), the second variation
$D^2_{(V,y,\tau)}F$ can be further simplified as
\begin{equation}
    \frac{1}{\sqrt{4\pi}^{n}}
    \int_{\Sigma}\Bigl(-\langle V,L^{\perp}V\rangle
    -2\tau\langle H,V\rangle-\tau^2|H|^2+\langle V,y\rangle
    -\frac{1}{2}|y^{\perp}|^2\Bigr)
    e^{-\frac{|X|^2}{4}}d\mu. \notag
\end{equation}
\end{proof}
In \cite{CM}, Colding and Minicozzi defined the following concept.

\begin{defn}\label{d3}
{\rm A critical point $(\Sigma,0,1)$ of $F$ is} {\it $F$-stable}
{\rm if for every compactly supported smooth variation $\Sigma_s$
with $\Sigma_0=\Sigma$ and $\frac{\partial\Sigma_s}{\partial
s}\big|_{s=0}=V$, there exist variations $x_s$ of $0$ and $t_s$ of
$1$ such that $D^2_{(V,y,\tau)}F\geq0$, where $y=\frac{\partial
x_s}{\partial s}\big|_{s=0}$ and $\tau=\frac{\partial t_s}{\partial
s}\big|_{s=0}$.}
\end{defn}
\begin{rem}
When $\Sigma$ is fixed, i.e. $V=0$, from (\ref{thm2}), we can see
that the second variation formula of $F$ is nonpositive under any
variations of $x_s$ and $t_s$.
\end{rem}
\section{An equivalent condition for F-stability}
Starting from this section, we assume that $\Sigma$ satisfies
$H=-\frac{1}{2}X^{\perp}$.
\subsection{Vector-valued eigenfunctions and eigenvalues of $L^{\perp}$}
From equation (\ref{thm2}), the second order operator $L^{\perp}$ is
the important term of second variation of $F$. When $\Sigma$ is a
hypersurface with $h=\frac{\langle X,\textbf{n}\rangle}{2}$, Colding
and Minicozzi \cite{CM} showed that the mean curvature function $h$
and the translations $\langle y,\textbf{n}\rangle$ are
eigenfunctions of $L$ with eigenvalues 1 and $\frac{1}{2}$,
respectively. Here $y$ is a constant vector in $\mathbb{R}^{n+1}$,
$\textbf{n}$ is the outer unit normal vector of $\Sigma$, and
\begin{align}
Lf=\Delta f+|A|^2f+\frac{1}{2}f-\frac{1}{2}\langle X,\nabla f\rangle. \notag
\end{align}
This property can also be generalized to the higher co-dimensional
case.
\begin{pro}\label{p4}
Assume that $\Sigma\subset\mathbb{R}^m$ is a smooth submanifold
satisfying $H=-\frac{X^{\perp}}{2}$, then the mean curvature vector
$H$ and the normal part $y^{\perp}$ of a constant vector field $y$
are vector-valued eigenfunctions of $L^{\perp}$ with
\begin{equation}
L^{\perp}H=H \ \ and \ \ L^{\perp}y^{\perp}=\frac{1}{2}y^{\perp}, \label{p401}
\end{equation}
where $L^{\perp}$ is as in (\ref{L}). Moreover, if $\Sigma$ is
compact, then $L^{\perp}$ is self-adjoint in the weighted space
defined in Definition \ref{d1} and
\begin{align}\label{p30}
&-\int_{\Sigma}\langle V_1,L^{\perp}V_2\rangle e^{-\frac{|X|^2}{4}}d\mu \notag \\
&=\int_{\Sigma}\Bigl(\langle\nabla^{\perp}V_1,\nabla^{\perp}V_2\rangle
-\langle A_{ij},V_1\rangle \langle A_{kl},V_2\rangle g^{ik}g^{jl}
-\frac{1}{2}\langle V_1,V_2\rangle\Bigr)e^{-\frac{|X|^2}{4}}d\mu.
\end{align}
\end{pro}
\begin{proof}
Fix $p\in\Sigma$ and choose an orthonormal frame $\{e_i\}$ such that
$\nabla_{e_i}e_j(p)=0$, $g_{ij}=\delta_{ij}$ in a neighborhood of
$p$. Using $H=-\frac{1}{2}X^{\perp}$, we have
\begin{align} \label{p3}
\nabla^{\perp}_{e_i}H=\nabla_{e_i}^{\perp}(-\frac{1}{2}X^{\perp})
=\frac{1}{2}\nabla^{\perp}_{e_i}(\langle X,e_j\rangle e_j-X)
=\frac{1}{2}\langle X,e_j\rangle A_{ij}.
\end{align}
In the second equality of (\ref{p3}), we used $X^{\top}=\langle
X,e_j\rangle e_j$. Taking another covariant derivative at $p$, it
gives
\begin{align}\label{p32}
\nabla^{\perp}_{e_k}\nabla^{\perp}_{e_i}H
&=\frac{1}{2}(\nabla_{e_k}\langle X,e_j\rangle)A_{ij}
+\frac{1}{2}\langle X,e_j\rangle\nabla^{\perp}_{e_k}A_{ij} \notag \\
&=\frac{1}{2}A_{ik}+\frac{1}{2}\langle X,A_{kj}\rangle A_{ij}
+\frac{1}{2}\langle X,e_j\rangle\nabla^{\perp}_{e_j}A_{ik},
\end{align}
where we used (\ref{p3}), $\nabla_{e_k}e_j(p)=0$, and the Codazzi
equation in the last equality. Taking the trace of (\ref{p32}) and
using $H=-\frac{1}{2}X^{\perp}$, we conclude that
\begin{align}
\Delta^{\perp}H
&=\frac{1}{2}H-\langle H,A_{ij}\rangle A_{ij}+\frac{1}{2}\nabla^{\perp}_{X^{\top}}H. \notag
\end{align}
Therefore,
\begin{align}
L^{\perp}H&=\Delta^{\perp}H+\langle A_{ij},H\rangle A_{ij}+\frac{1}{2}H
-\frac{1}{2}\nabla^{\perp}_{X^{\top}}H=H. \notag
\end{align}
For a constant vector $y$ in $\mathbb{R}^m$, the covariant
derivative of $y^{\perp}$ is
\begin{align}\label{p34}
\nabla^{\perp}_{e_i}y^{\perp}
=\nabla^{\perp}_{e_i}(y-\langle y,e_j\rangle e_j)
=-\langle y,e_j\rangle A_{ij}.
\end{align}
Taking another covariant derivative at $p$, it gives
\begin{align}\label{p35}
\nabla^{\perp}_{e_k}\nabla^{\perp}_{e_i}y^{\perp}
&=-(\nabla_{e_k}\langle y,e_j\rangle)A_{ij}-\langle y,e_j\rangle\nabla^{\perp}_{e_k}A_{ij} \notag \\
&=-\langle y,A_{kj}\rangle A_{ij}-\langle y,e_j\rangle \nabla^{\perp}_{e_j}A_{ki},
\end{align}
by $\nabla_{e_k}e_j(p)=0$ and the Codazzi equation. Taking the trace
of (\ref{p35}) and using (\ref{p3}), (\ref{p34}), we conclude that
\begin{align}
\Delta^{\perp}y^{\perp}&=-\langle y,A_{ij}\rangle A_{ij}-\langle y,e_j\rangle\nabla^{\perp}_{e_j}H \notag \\
&=-\langle y^{\perp},A_{ij}\rangle A_{ij}-\frac{1}{2}\langle y,e_j\rangle\langle X,e_i\rangle A_{ij} \notag \\
&=-\langle y^{\perp},A_{ij}\rangle A_{ij}+\frac{1}{2}\langle X,e_i\rangle\nabla^{\perp}_{e_i}y^{\perp} \notag \\
&=-\langle y^{\perp},A_{ij}\rangle A_{ij}+\frac{1}{2}\nabla^{\perp}_{X^{\top}}y^{\perp}. \notag
\end{align}
Therefore,
\begin{align}
L^{\perp}y^{\perp}=\Delta^{\perp}y^{\perp}+\langle A_{ij},y^{\perp}\rangle A_{ij}+\frac{1}{2}y^{\perp}
-\frac{1}{2}\nabla^{\perp}_{X^{\top}}y^{\perp}=\frac{1}{2}y^{\perp}. \notag
\end{align}
The equation (\ref{p30}) follows from the divergence theorem that
\begin{align}
\int_{\Sigma} \langle\Delta^{\perp}V_1,V_2\rangle e^{-\frac{|X|^2}{4}}d\mu
+\int_{\Sigma} \langle\nabla^{\perp}V_1,\nabla^{\perp}V_2\rangle e^{-\frac{|X|^2}{4}}d\mu
-\frac{1}{2}\int_{\Sigma}\langle\nabla_{X^{\top}}^{\perp}V_1,V_2\rangle e^{-\frac{|X|^2}{4}}d\mu=0. \notag
\end{align}
\end{proof}
In the case that $\Sigma$ is compact, since the operator $L^{\perp}$
is self-adjoint in the weighted inner product space with respect to
$\Sigma$ and $L^{\perp}H=H$,
$L^{\perp}y^{\perp}=\frac{1}{2}y^{\perp}$, we have
\begin{align}
\langle H,y^{\perp}\rangle_e=\langle L^{\perp}H,y^{\perp}\rangle_e
=\langle H,L^{\perp}y^{\perp}\rangle_e=\frac{1}{2}\langle H,y^{\perp}\rangle_e. \notag
\end{align}
Hence $\langle H,y^{\perp}\rangle_e=0$, for any constant vector $y$.
Since $\langle H,y^{\perp}\rangle_e=\langle H,y\rangle_e$, it gives
\begin{align}
\int_{\Sigma}He^{-\frac{|X|^2}{4}}d\mu=\overrightarrow{0}. \label{p41}
\end{align}
When $\Sigma$ is complete, (\ref{p41}) is still true provided that
$\Sigma$ has polynomial volume growth and the second fundamental
form of $\Sigma$ is of polynomial growth.
\subsection{An equivalent condition}
In the following theorems, we give an equivalent condition for
$F(\Sigma,0,1)$ to be stable. It is inspired by the proof of Lemma
4.23 of Colding and Minicozzi in \cite{CM}.

\begin{thm}\label{thm3}
Suppose $\Sigma \subset \mathbb{R}^m$ is an $n$-dimensional smooth
closed self-shrinker, $H=-\frac{X^{\perp}}{2}$. The following
statements are equivalent:

(i) $\Sigma$ is F-stable.

(ii) $\int_{\Sigma}\langle V,-L^{\perp}V\rangle
e^{-\frac{|X|^2}{4}}d\mu\geq0$ for any smooth normal vector field
$V$ which satisfies
\begin{align} \label{t301}
\int_{\Sigma}\langle V,H\rangle e^{-\frac{|X|^2}{4}}d\mu=0 \quad\mbox{and}\quad
\int_{\Sigma}Ve^{-\frac{|X|^2}{4}}d\mu=\overrightarrow{0}.
\end{align}
\end{thm}
\begin{proof}
$(i) \Rightarrow (ii)$

Assume the contrary that there is a smooth normal vector field $V$
satisfying (\ref{t301}) but with $\int_{\Sigma}\langle
V,-L^{\perp}V\rangle e^{\frac{-|X|^2}{4}}d\mu<0$. For any real value
$\tau$ and constant vector $y$ in $\mathbb{R}^m$, using
(\ref{thm2}), we have
\begin{align}
    &D^2_{(V,y,\tau)}F \notag \\
    =&\frac{1}{\sqrt{4\pi}^{n}}\int_{\Sigma}\left(
    -\langle V,L^{\perp}V\rangle-2\tau\langle H,V\rangle
    -\tau^2|H|^2+\langle V,y\rangle-\frac{1}{2}|y^{\perp}|^2\right)e^{-\frac{|X|^2}{4}}d\mu \notag \\
    =&\frac{1}{\sqrt{4\pi}^{n}}\int_{\Sigma}\left(-\langle V,L^{\perp}V\rangle-\tau^2|H|^2-
    \frac{1}{2}|y^{\perp}|^2\right)
    e^{-\frac{|X|^2}{4}}d\mu \notag\\
    <&0, \notag
\end{align}
where the second equality follows from the conditions (\ref{t301}).
This contradicts the stability of $F$.

$(ii) \Rightarrow (i)$

The space
\begin{align}
N_{tr}=\{y^{\perp}|y \in \mathbb{R}^{m}\} \notag
\end{align}
is a Hilbert space with the weighted inner product that is spanned
by $E_1^{\perp},...,E_m^{\perp}$ where $\{E_i\}$ is the standard
basis in $\mathbb R^m$. Given a smooth normal vector field $V$, it
can be decomposed as $aH+z^{\perp}+V_0$. Here $aH$ and $z^{\perp}$
are the projections of $V$ to $H$ and $N_{tr}$, respectively. Note
that $V_0$ is a smooth normal vector field satisfying (\ref{t301}).
For any real value $\tau$ and constant vector $y\in\mathbb{R}^m$, by
plugging the decomposition of $V$ into (\ref{thm2}), we have
\begin{align}
&D^2_{(V,y,\tau)}F \notag \\
=&\frac{1}{\sqrt{4\pi}^{n}}\int_{\Sigma}\left(
    -\langle V,L^{\perp}V\rangle-2\tau\langle H,V\rangle
    -\tau^2|H|^2+\langle V,y\rangle-\frac{1}{2}|y^{\perp}|^2\right)e^{-\frac{|X|^2}{4}}d\mu \notag  \\
   =&\frac{1}{\sqrt{4\pi}^{n}}\int_{\Sigma}\Big(-a^2|H|^2-\frac{1}{2}|z^{\perp}|^2-
   \langle V_0,L^{\perp}V_0\rangle-2\tau a|H|^2
-\tau^2|H|^2 \notag \\
&+\langle z^{\perp},y^{\perp}\rangle-\frac{1}{2}|y^{\perp}|^2\Big)e^{-\frac{|X|^2}{4}}d\mu  \notag \\
\geq &\frac{1}{\sqrt{4\pi}^{n}}\int_{\Sigma}\left(-|H|^2(a+\tau)^2-\frac{1}{2}|z^{\perp}-y^{\perp}|^2\right)
e^{-\frac{|X|^2}{4}}d\mu, \notag
\end{align}
 where the condition (ii) is used in the last inequality. Choosing $\tau=-a$ and
$y=z$, it gives $D^2_{(V,z,-a)}F\geq0$. That is, $\Sigma$ is
$F$-stable.
\end{proof}
For the complete case, we define
\begin{align}
S_{\Sigma}=\{V\in N\Sigma\big|\ \ |V|(X) \mbox{ and } |\nabla^{\perp}V|(X)
\mbox{ are of polynomial growth } \}. \notag
\end{align}
Note that $V\in S_{\Sigma}$ might not be of compact support. We can
also find the following equivalent condition for the stability of
$F$ in the complete case.
\begin{thm}\label{t4}
Let $\Sigma \subset \mathbb{R}^m$ be an $n$-dimensional smooth
complete self-shrinker, $H=-\frac{X^{\perp}}{2}$, without boundary.
Suppose that the second fundamental form $A$ of $\Sigma$ is of
polynomial growth and $\Sigma$ has polynomial volume growth. The
following statements are equivalent:

(i) $\Sigma$ is $F$-stable.

(ii) $\int_{\Sigma}\langle V,-L^{\perp}V\rangle
e^{-\frac{|X|^2}{4}}d\mu\geq0$ for any smooth normal vector field
$V$ in $S_{\Sigma}$ which satisfies
\begin{align}
\int_{\Sigma}\langle V,H\rangle e^{-\frac{|X|^2}{4}}d\mu=0 \quad \mbox{and} \quad
\int_{\Sigma}Ve^{-\frac{|X|^2}{4}}d\mu=\overrightarrow{0}. \notag
\end{align}
\end{thm}
\begin{rem}
When $V\in S_{\Sigma}$, $A$ is of polynomial growth, and $\Sigma$
has polynomial volume growth, the integral
$\int_{\Sigma}\left(|\nabla^{\perp}V|^2-|\langle
A,V\rangle|^2-\frac{1}{2}|V|^2\right)e^{-\frac{|X|^2}{4}}d\mu$ is
finite. By divergence theorem, we have
\begin{align}
\langle V,-L^{\perp}V\rangle_e=\int_{\Sigma}\left(|\nabla^{\perp}V|^2-|\langle
A,V\rangle|^2-\frac{1}{2}|V|^2\right)e^{-\frac{|X|^2}{4}}d\mu. \label{r2}
\end{align}
\end{rem}
\begin{proof}[Proof of Theorem \ref{t4}]
$(i) \Rightarrow (ii)$

Assume the contrary that there is a smooth normal vector field $V$
in $S_{\Sigma}$ satisfying
\begin{align}\label{thm4}
\langle V,H\rangle_e=0, \quad
\int_{\Sigma}Ve^{\frac{-|X|^2}{4}}d\mu=\overrightarrow{0}, \quad
\mbox{and } \langle V,-L^{\perp}V\rangle_e<0.
\end{align}
Here $V$ may not have a compact support. For $j\in\mathbb{N}$,
consider smooth functions $\phi_j:\mathbb R^+\bigcup\{0\}\to\mathbb
R$ that satisfy $0\leq\phi_j\leq 1$, $\phi_j\equiv1$ on $[0,j)$,
$\phi_j\equiv0$ outside $[0,j+2)$ and $|\phi'_j|\leq1$. Define
cutoff functions $\psi_j(X)=\phi_j(\rho(X))$, $X\in\Sigma$, where
$\rho(X)$ is the distance function from a fixed point $p\in \Sigma$
to $X$ with respect to the metric $g_{ij}$. Let
$V_j(X)=\psi_j(X)V(X)$, then we have
\begin{align}
|\nabla^{\perp}V_j|^2
&=\underset{i=1}{\overset{n}{\sum}}|(\nabla_{e_i}\psi_j)V+\psi_j\nabla_{e_i}^{\perp}V|^2 \notag \\
&\leq 2|\nabla\psi_j|^2|V|^2+2|\psi_j|^2|\nabla^{\perp}V|^2 \notag \\
&\leq 2|V|^2+2|\nabla^{\perp}V|^2. \notag
\end{align}
Here $\{e_i\}$ is an orthonormal basis for $T_X\Sigma$. Using
(\ref{r2}), (\ref{thm4}), and the dominant convergence theorem, it
follows that
\begin{align}
\lim\limits_{j\to \infty}\langle V_j,-L^{\perp}V_j\rangle_e
=\langle V,-L^{\perp}V\rangle_e \mbox{ and}
\lim\limits_{j\to \infty}\langle V_j,H\rangle_e
=\lim\limits_{j\to\infty}\langle V_j,y^{\perp}\rangle_e=0. \notag
\end{align}
For any small positive $\epsilon$, choose a sufficiently large $j$
such that
\begin{align}
&\langle V_j,-L^{\perp}V_j\rangle_e<\frac{1}{2}\langle V,-L^{\perp}V\rangle_e<0,\notag \\
&|\langle V_j,H\rangle_e|<\epsilon |H|_e,\quad\mbox{and}\quad
\max\limits_{|y^{\perp}|_e=1}|\langle V_j,y^{\perp}\rangle_e|<\epsilon. \notag
\end{align}
For any real value $\tau$ and constant vector $y$ in $\mathbb R^m$,
we get
\begin{align}
    &D^2F_{(V_j,y,\tau)} \notag \\
    =&\frac{1}{\sqrt{4\pi}^{n}}\int_{\Sigma}\left(
    -\langle V_j,L^{\perp}V_j\rangle-2\tau\langle H,V_j\rangle
    -\tau^2|H|^2+\langle V_j,y^{\perp}\rangle-\frac{1}{2}|y^{\perp}|^2\right)
    e^{-\frac{|X|^2}{4}}d\mu \notag \\
    <&\frac{1}{\sqrt{4\pi}^{n}}\left(
    -\frac{1}{2}\langle V,L^{\perp}V\rangle_e
    +2\tau\epsilon|H|_e-\tau^2|H|^2_e+\epsilon|y^{\perp}|_e-\frac{1}{2}|y^{\perp}|_e^2
    \right) \notag\\
    =&\frac{1}{\sqrt{4\pi}^{n}}\left(
    -\frac{1}{2}\langle V,L^{\perp}V\rangle_e
    +\epsilon^2-(\tau|H|_e-\epsilon)^2+\frac{1}{2}\epsilon^2-\frac{1}{2}(|y^{\perp}|_e-\epsilon)^2
    \right). \notag
\end{align}
Choosing $\epsilon^2<\frac{1}{10}\langle V,L^{\perp}V\rangle_e$, we
get $D^2F_{(V_j,y,\tau)}<0$ for every $\tau$ and $y$. This
contradicts the stability of $F$.

$(ii) \Rightarrow (i)$

 A compactly supported smooth normal vector field $V$ can be
decomposed as $aH+z^{\perp}+V_0$, where $V_0,$ $H$, and $N_{tr}$ are
mutually orthogonal with respect to the weighted inner product.
Because $V$, $H$, and $z^{\perp}$ belong to $S_{\Sigma}$ and
$S_{\Sigma}$ is a linear vector space, $V_0$ belongs to
$S_{\Sigma}$, too. The remaining part of the proof is essentially
the same as the proof of $(ii)\Rightarrow(i)$ in Theorem \ref{thm3}.
\end{proof}
We immediately have the following corollaries.
\begin{cor}\label{c1}
Suppose $\Sigma_i^{n_i}\subset\mathbb R^{m_i}$, $i=1,2$, are smooth
closed self-shrinkers which satisfy $H_i=-\frac{X_i^{\perp}}{2}$,
where $X_i$ are the position vectors of $\Sigma_i$. Then
$\Sigma_1\times \Sigma_2\subset\mathbb R^{m_1+m_2}$ is a
self-shrinker and is F-unstable.
\end{cor}
\begin{proof}
The mean curvature $H$ of $\Sigma_1\times\Sigma_2$ is expressed as
$(H_1,H_2)\in\mathbb R^{m_1}\times \mathbb R^{m_2}$ and
$\Sigma_1\times \Sigma_2$ is a self-shrinker because
$H_1=-\frac{X_1^{\perp}}{2}$ and $H_2=-\frac{X_2^{\perp}}{2}$. To
prove this corollary, by Theorem \ref{thm3}, it suffices to
construct a smooth normal vector field $V$ such that (\ref{t301})
holds while $\int_{\Sigma}\langle V,-L^{\perp}V\rangle
e^{-\frac{|X|^2}{4}}d\mu<0$. Let $V=(aH_1,bH_2)$, where $a$ and $b$
would be chosen later. Note that $V$ is not vanish since $\Sigma_1$
and $\Sigma_2$ are closed submanifolds in Euclidean spaces. The
first integral in (\ref{t301}) is
\begin{align}
&\int_{\Sigma_1\times\Sigma_2}\langle V,H\rangle e^{-\frac{|X|^2}{4}}d\mu \notag \\
=&\int_{\Sigma_1}\int_{\Sigma_2}(a|H_1|^2+b|H_2|^2) e^{-\frac{|X_1|^2}{4}}
e^{-\frac{|X_2|^2}{4}}d\mu_2d\mu_1 \notag \\
=&a\int_{\Sigma_1}|H_1|^2e^{-\frac{|X_1|^2}{4}}d\mu_1\int_{\Sigma_2}e^{-\frac{|X_2|^2}{4}}d\mu_2
+b\int_{\Sigma_1}e^{-\frac{|X_1|^2}{4}}d\mu_1\int_{\Sigma_2}|H_2|^2e^{-\frac{|X_2|^2}{4}}d\mu_2. \notag
\end{align}
We can choose $a$ and $b$ to be nonzero constants such that
$\int_{\Sigma_1\times\Sigma_2}\langle V,H\rangle
e^{-\frac{|X^2|}{4}}d\mu=0$. The second integral
$\int_{\Sigma_1\times\Sigma_2}Ve^{-\frac{|X|^2}{4}}d\mu$ in
(\ref{t301}) is equal to $\overrightarrow{0}$ because of the
equation (\ref{p41}). The weighted inner product $\langle
V,-L^{\perp}V\rangle _e$ can be computed as
\begin{align}
&\int_{\Sigma_1\times\Sigma_2}\langle V,-L^{\perp}V\rangle e^{-\frac{|X|^2}{4}}d\mu \notag \\
=&\int_{\Sigma_1\times\Sigma_2}\langle (aH_1,bH_2),-(aH_1,bH_2)\rangle e^{-\frac{|X|^2}{4}}d\mu \notag \\
=&-a^2\int_{\Sigma_1}|H_1|^2e^{-\frac{|X_1|^2}{4}}d\mu_1\int_{\Sigma_2}e^{-\frac{|X_2|^2}{4}}d\mu_2
-b^2\int_{\Sigma_1}e^{-\frac{|X_1|^2}{4}}d\mu_1\int_{\Sigma_2}|H_2|^2e^{-\frac{|X_2|^2}{4}}d\mu_2 \notag \\
<&0. \notag
\end{align}
Here the first equality follows from the fact that $L^{\perp}$
splits to $L^{\perp}_1$ and $L^{\perp}_2$, and the equation
(\ref{p401}).
\end{proof}
\begin{cor} \label{c2}
Let $\Sigma_i^{n_i}\subset\mathbb R^{m_i}$, $i=1,2$, be two smooth
complete self-shrinkers without boundary which satisfy
$H_i=-\frac{X_i^{\perp}}{2}\neq0$, where $X_i$ is the position vector
of $\Sigma_i$. Suppose that each $\Sigma_i$ has polynomial volume
growth and the second fundamental form of each $\Sigma_i$ is of
polynomial growth. Then $\Sigma_1\times \Sigma_2\subset\mathbb
R^{m_1+m_2}$ is a self-shrinker and is F-unstable.
\end{cor}
Using Theorem \ref{t4}, the proof of Corollary \ref{c2} is similar
to the proof of Corollary \ref{c1}.
\section{The unstability of Anciaux's examples}\label{Sec4}
\subsection{Anciaux's examples}
Let $\langle\langle\cdot,\cdot\rangle\rangle
=\overset{n}{\underset{i=1}{\sum}}dz_i\otimes d\overline{z}_i$ be
the standard Hermitian form on $\mathbb{C}^n$, where
$z_i=x_i+\sqrt{-1}y_i$, $i=1,...,n$ are the standard complex
coordinates. The standard Riemannian metric is
$\langle\cdot,\cdot\rangle=\mbox{Re}\langle\langle\cdot,\cdot\rangle\rangle
=\overset{n}{\underset{i=1}{\sum}}(dx_i^2+dy_i^2)$ and the
symplectic form is
$\omega(\cdot,\cdot)=-\mbox{Im}\langle\langle\cdot,\cdot\rangle\rangle
=\overset{n}{\underset{i=1}{\sum}}dx_i\wedge dy_i$. We have
$\omega(\cdot,\cdot)=\langle J\cdot,\cdot\rangle$, where $J$ is the
standard almost complex structure $J(\frac{\partial}{\partial
x_i})=\frac{\partial}{\partial y_i}$ and $J(\frac{\partial}{\partial
y_i})=-\frac{\partial}{\partial x_i}$.

Recall that an immersion $\psi $ from a manifold $M$ of dimension
$(n-1)$ into $\mathbb{S}^{2n-1}$\ is said to be $Legendrian$ if
$\alpha |_{\psi (M)}=0$ for the contact 1-form $\alpha (\cdot
)=\omega (X^M,\cdot)$, where $X^M$ is the position vector and
$\omega$ is the standard symplectic form on $\mathbb{C}^n$.
Moreover, $d\alpha =2\omega $ and $\langle Jy,z\rangle=\omega
(y,z)=\frac{1}{2}d\alpha (y,z)=0,$ $\langle JX^M,y\rangle=\omega
(X^M,y)=\alpha (y)=0$ for all $y,$ $z\in T\psi (M)$. It means that
$y$, $Jz$, $X^M$, and $JX^M$ are mutually orthogonal with respect to
the standard metric $g$ for any $y,$ $z\in T\psi (M)$. When $\psi$
is a minimal immersion, the complex scalar product $\gamma\psi$ of a
smooth regular curve $\gamma:I\to\mathbb C^*$ and $\psi$ is a
Lagrangian submanifold in $\mathbb C^n$, i.e.,
$\omega|_{\gamma\psi}\equiv0$. This was observed by Anciaux in
\cite{An}. Indeed, he proved by following Lemma.
\begin{lem}\cite{An} \label{l2}
Let $\psi : M\rightarrow \mathbb S^{2n-1}$ be a minimal Legendrian
immersion and $\gamma:I\rightarrow \mathbb C^{*}$ be a smooth
regular curve parameterized by the arclength $s$. Then the following
immersion
\begin{align}
\left.\begin{array}{ll}
\gamma \ast \psi :I\times M &\to \mathbb C^{n} \\
\ \ \ \ \ \ \ \ \ \ (s,\sigma)&\to\gamma (s) \psi (\sigma ) \notag
\end{array}\right.
\end{align}
is a Lagrangian. Moreover, $\gamma \ast \psi$ satisfies the
self-shrinker equation
\begin{align}
H+\frac{1}{2}(\gamma \ast \psi )^{\perp }=0 \notag
\end{align}
if and only if $\gamma$ satisfies the following system of ordinary
differential equations:
\begin{align}
\left\{\begin{array}{ll}
\ \ \ \ \ \ \ \ \ r'(s)&=\cos(\theta-\phi), \\
\theta'(s)-\phi'(s)&=(\frac{r}{2}-\frac{n}{r})\sin(\theta-\phi), \label{l201}
\end{array}\right.
\end{align}
where the curve $\gamma$ is denoted as $r(s)e^{i\phi(s)}$ and
$\theta$ is the angle of the tangent and the x-axis. From
(\ref{l201}), we have a conservation law
\begin{align}
r^ne^{-\frac{r^2}{4}}\sin(\theta-\phi)=E, \label{l202}
\end{align}
where $0< E\leq E_{max}=(\frac{2n}{e})^{n/2}$ is a constant
determined by the initial data $(r(s_0),\theta(s_0)-\phi(s_0))$.
\end{lem}
\subsection{The unstability for general variations } \label{us42}
Because the complete noncompact Lagrangian examples constructed by
Anciaux in \cite{An} do not have polynomial volume growth, the
$F$-functional is not well-defined and hence we will only discuss
the closed cases. That is, the corresponding curves $\gamma$ are
closed and the immersions $\psi:M\to \mathbb{S}^{2n-1}$ are closed.
\begin{thm} \label{t5}
Fix $n\geq 2.$ Let $\Sigma $ be the image of the immersion $\gamma
(s)\ast \psi (\sigma )$ in Lemma \ref{l2}. If $\Sigma$ is closed,
then $\Sigma $ is $F$-unstable.
\end{thm}

To prove the result, we first set up the notations and derive a few
Lemmas. For a fixed point $p\in \Sigma=\gamma \ast \psi(I\times M)$,
it can be represented by $\gamma(s_0) q$ for some $s_0\in I$ and $q
\in \psi(M)$. Choose a local normal coordinate system
$x^1,...,x^{n-1}$ at $q$. Denote $u_s=\frac{\partial X}{\partial
s}=\gamma'X^M$, $e_i=\frac{\partial X^M}{\partial x^i}$, and
$u_i=\frac{\partial X}{\partial x^i}=\gamma e_i$ for $i=1,...,n-1$,
where $X^M$ is the position vector of $\psi(M)$ and $X=\gamma X^M$.
The matrix $(g_{\alpha\beta})$ of the induced metric of $\Sigma$
with respect to the basis $u_1,...,u_{n-1},u_s$ is
\begin{align}
g_{ss}=1, \quad g_{js}=g_{sj}=0, \quad g_{jk}=r^2h_{jk},\quad
\mbox{and}\quad h_{jk}(q)=\delta_{jk} \label{l32}
\end{align}
for $j,k=1,...,n-1$. The Levi-Civita connections on $\Sigma$ and
$\psi(M)$ are denoted by $\nabla$ and $\nabla^M$, respectively.
 Define
\begin{align}
N_0=\{V|V=J(\gamma w), ~w\in \Gamma (T\psi(M))\}. \notag
\end{align}
For $V\in N_0$, the operator $\langle V,-L^{\perp}V\rangle_e$ can be
simplified as below.
\begin{lem}\label{l3}
Assume that $\Sigma$ is a closed Lagrangian self-shrinker as in
Lemma~\ref{l2} and $V\in N_0$ is represented by $J(\gamma w)$. The
second fundamental forms of $\Sigma $\ in $\mathbb C^{n}$ and
$\psi(M)$ in $\mathbb S^{2n-1}$ are denoted by $A^{\Sigma}$ and
$A^{M,\mathbb{S}}$, respectively. Then we have
\begin{align}
(i) ~~~ &|\langle A^{\Sigma},V\rangle|^{2}=|\langle A^{M,\mathbb S},Jw\rangle|^{2}
+2\sin^{2}(\theta -\phi )|w|^{2},\label{l301} \\
(ii) ~~ &|\nabla ^{\perp }V|^{2}=|\nabla ^{M}w|^{2}+2\cos ^{2}(\theta
-\phi )|w|^{2},\label{l302} \\
(iii) ~&\langle V,-L^{\perp}V\rangle_e
=-\int_{\gamma}\left(\frac{1}{2}r^{2}-2+4\sin ^{2}
 (\theta -\phi)\right)e^{\frac{-r^2}{4}}r^{n-1}ds\int_M|w|^{2}d\mu_M \notag \\
&~~~~~~~~~~~~~~~+\int_{\gamma}e^{\frac{-r^2}{4}}r^{n-1}ds
\int_M\left(|\nabla ^{M}w|^{2}-|\langle A^{M,\mathbb S},Jw\rangle|^{2}\right)d\mu_M.\label{l303}
 \end{align}
\end{lem}
\begin{proof}
(i) For $V\in N_0$, it can be represented by $J(\gamma w)$ for some
vector field $w\in \Gamma(T\psi(M))$. Using
$\gamma\overline{\gamma}=r^2$ and
$\gamma'\overline{\gamma}=re^{i(\theta -\phi)}$, we conclude that
\begin{align}
&\langle A^{\Sigma}_{kl},V\rangle
=\mbox{Re}\langle\langle\gamma\frac{\partial^2X^M}{\partial x^k\partial x^l},J(\gamma w)\rangle\rangle
=r^{2}\mbox{Re}\langle\langle A_{kl}^{M},Jw\rangle\rangle
=r^{2}\langle A_{kl}^{M,S},Jw\rangle, \notag \\
&\langle A^{\Sigma}_{ks},V\rangle
=\mbox{Re}\langle\langle \gamma'\frac{\partial X^M}{\partial x^k}, J(\gamma w)\rangle\rangle
=r\sin(\theta -\phi )\langle e_k,w\rangle, \label{l33}\\
&\langle A^{\Sigma}_{ss},V\rangle
=\mbox{Re}\langle\langle \gamma''X^M,J(\gamma w)\rangle\rangle
=\mbox{Re}(\gamma''\overline{\gamma}\langle\langle X^M, Jw\rangle\rangle)
=0 \notag
\end{align}
for $k,l=1,..,n-1$. Here the second equalities of the second and
third equations of (\ref{l33}) are followed by the fact that $e_k$,
$Jw$, $X^M$, and $JX^M$ are mutually orthogonal. Combining
(\ref{l32}) and (\ref{l33}), it gives
\begin{align}
|\langle A^{\Sigma},V\rangle|^{2}
=&\overset{n-1}{\underset{k,l=1}{\sum}}\langle A^{\Sigma}_{kl},V\rangle^{2}\frac{1}{r^{4}}
+2\underset{k=1}{\overset{n-1}{\sum}}\langle A^{\Sigma}_{ks},V\rangle^{2}\frac{1}{r^{2}}+
\langle A^{\Sigma}_{ss},V\rangle^{2}
\notag \\
=&|\langle A^{M,\mathbb S},Jw\rangle|^{2}+2\sin ^{2}(\theta -\phi )|w|^{2} \ \ \ \ \ \mbox{at $p$}. \notag
\end{align}
(ii) Since $\Sigma$ is a Lagrangian,
$\{Ju_\alpha\}_{\alpha=1,...,n-1,s}$ is an orthogonal basis at $p$
for the normal bundle. We will calculate the normal projection of
$(\nabla^{\perp} _{u_\alpha }J(\gamma w))_{\alpha=1,...,n-1,s}$ on
$Ju_{j}$ and $Ju_{s}.$ Using the property that $w$, $Je_k$, $X^M$,
and $JX^M$ are mutually orthogonal, $\gamma\overline{\gamma}=r^2$
and $\gamma'\overline{\gamma}=re^{i(\theta -\phi)}$, we conclude
that
\begin{align} \label{l34}
\hspace{-0.35cm}\left.\begin{array}{ll}
&\langle \nabla^{\perp} _{u_k}J(\gamma w),Ju_{j}\rangle=
\mbox{Re}\langle\langle i\gamma\frac{\partial}{\partial x^k} w,i\gamma e_j\rangle\rangle
=r^{2}\langle \nabla _{e_k}^{M}w,e_{j}\rangle \\[.1cm]
&\langle \nabla^{\perp} _{u_k}J(\gamma w),Ju_s\rangle
=-\mbox{Re}\langle\langle i\gamma w,\frac{\partial}{\partial x^k}i\gamma'X^M\rangle\rangle
=-r\cos(\theta -\phi)\langle w,e_k\rangle \\[.1cm]
&\langle\nabla^{\perp}_{u_s}J(\gamma w),Ju_{j}\rangle
=\mbox{Re}\langle\langle i\gamma'w,i\gamma e_{j}\rangle\rangle
=r\cos (\theta -\phi)\langle w,e_j\rangle \\[.1cm]
&\langle \nabla^{\perp}_{u_s}J(\gamma w),Ju_s\rangle
=\mbox{Re}\langle\langle i\gamma' w,i\gamma'X^M\rangle\rangle=0. \\
\end{array}\right.
\end{align}
From (\ref{l34}), it follows that
\begin{align}
&|\nabla ^{\perp }V|^{2}
=\langle\nabla^{\perp} _{u_{\alpha}}J(\gamma w),\nabla^{\perp} _{u_{\beta}}J(\gamma w)\rangle
g^{\alpha \beta } \notag \\
=&\overset{n-1}{\underset{k=1}{\sum}}\langle\nabla^{\perp}_{u_k}J(\gamma w),\nabla^{\perp}_{u_k}J(\gamma
w)\rangle\frac{1}{r^{2}}+\langle\nabla _{u_s}J(\gamma w),\nabla _{u_s}J(\gamma w)\rangle \notag \\
=&\left(\overset{n-1}{\underset{j,k=1}{\sum}}\langle\nabla^{\perp}
_{u_k}J(\gamma w),\frac{Ju_{j}}{r}\rangle^2
+\overset{n-1}{\underset{k=1}{\sum}}\langle\nabla^{\perp}_{u_k}J(\gamma
w),Ju_s\rangle^2\right)\frac{1}{r^{2}}
+\overset{n-1}{\underset{j=1}{\sum}}\langle\nabla^{\perp}_{u_s}J(\gamma w),\frac{Ju_j}{r}\rangle^2 \notag \\
=&\overset{n-1}{\underset{j,k=1}{\sum}}\langle\nabla_{e_k}^M w,e_j\rangle^2
+\overset{n-1}{\underset{j=1}{\sum}}2\cos^2(\theta-\phi)\langle w,e_j\rangle^2 \notag \\
=&|\nabla ^{M}w|^{2}+2\cos ^{2}(\theta -\phi )|w|^{2}. \notag
\end{align}
(iii) Plugging (\ref{l301}) and (\ref{l302}) into (\ref{r2}), and
using
$e^\frac{-|X|^2}{4}d\mu_{\Sigma}=e^{-\frac{r^2}{4}}r^{n-1}dsd\mu_M$,
we get
\begin{align}
&\langle V, -L^{\perp}V\rangle_e \notag \\
=&\int_{\Sigma}\left(|\nabla^{\perp}V|^2-|\langle A^{\Sigma}, V\rangle|^2-\frac{1}{2}|V|^2\right)
e^{-\frac{|X|^2}{4}}d\mu_{\Sigma} \notag \\
=&\int_{\gamma}\int_M\Big(|\nabla ^{M}w|^{2}+2\cos ^{2}(\theta
-\phi )|w|^{2}-\left(|\langle A^{M,\mathbb S},Jw\rangle|^{2}
+2\sin^{2}(\theta -\phi )|w|^{2}\right) \notag \\
&-\frac{1}{2}r^2|w|^2\Big) e^{-\frac{r^2}{4}}r^{n-1}d\mu_mds \notag \\
=&-\int_{\gamma}\left(\frac{1}{2}r^{2}-2+4\sin ^{2}
 (\theta -\phi)\right)e^{\frac{-r^2}{4}}r^{n-1}ds\int_M|w|^{2}d\mu_M \notag \\
&+\int_{\gamma}e^{\frac{-r^2}{4}}r^{n-1}ds
\int_M\left(|\nabla ^{M}w|^{2}-|\langle A^{M,\mathbb S},Jw\rangle|^{2}\right)d\mu_M.\notag
\end{align}
Thus (iii) is proved.
\end{proof}
To further simplify $\langle V,-L^{\perp}V\rangle_e$, we now derive
some integral properties of the curve $\gamma$.
\begin{lem}
Let $\gamma:I\to\mathbb C^*$ be a closed smooth regular curve
parameterized by the arclength $s$ satisfying (\ref{l201}). That is,
$\gamma*\psi$ in Lemma \ref{l2} define a closed self-shrinker. Then
one has
\begin{align}
\int_{\gamma}(\frac{1}{2}r^2-n)r^{n-1}e^{-\frac{r^2}{4}}ds=0 \label{p501}
\end{align}
and
\begin{align}
\int_{\gamma}(\frac{1}{2r^2}-\frac{n}{r^4})r^{n-1}e^{-\frac{r^2}{4}}ds
=-\int_{\gamma}\frac{4\cos^2(\theta-\phi)}{r^4}r^{n-1}e^{-\frac{r^2}{4}}ds. \label{p502}
\end{align}
\end{lem}
\begin{rem}
The equality (\ref{p501}) is used to simplify (\ref{l303}) while the
equality (\ref{p502}) is used to simplify (\ref{l403}) for
Lagrangian variation.
\end{rem}
\begin{proof}
Equation (\ref{p501}) follows from the simplification of equation
(\ref{p201}) and $\int_Md\mu_M\neq0$. Indeed, equation (\ref{p201})
becomes
\begin{align}
0=\int_{\gamma}\int_{M}(r^2-2n)e^{-\frac{r^2}{4}}r^{n-1}d\mu_Mds
=\int_{\gamma}(r^2-2n)e^{-\frac{r^2}{4}}r^{n-1}ds\int_{M}d\mu_M. \notag
\end{align}
Recall that the linear operator $\mathcal{L}f=\Delta
f-\frac{1}{2}\langle X, \nabla f\rangle
=e^{\frac{|X|^2}{4}}div(e^{-\frac{|X|^2}{4}}\nabla f)$ in
Proposition \ref{p1}. It gives
\begin{align}
\int_{\Sigma}\mathcal{L}(\frac{1}{|X|^2})e^{-\frac{|X|^2}{4}}d\mu_{\Sigma}
=\int_{\Sigma}div(e^{-\frac{|X|^2}{4}}\nabla\frac{1}{|X|^2})d\mu_{\Sigma}
=0 \label{p51}
\end{align}
since $\partial\Sigma=\emptyset$. On the other hand, using equation
(\ref{p102}) and $\nabla|X|^2=2X^{\top}$ gives
\begin{align}
\mathcal{L}(\frac{1}{|X|^2})=\frac{-\mathcal{L}|X|^2}{|X|^4}+\frac{2|\nabla|X|^2|^2}{|X|^6}
=\frac{-2n+|X|^2}{|X|^4}+\frac{8|X^{\top}|^2}{|X|^6}. \label{p52}
\end{align}
Combining (\ref{p51}), (\ref{p52}), and using
$|X^{\top}|=\mbox{Re}\left(re^{i(\phi-\theta)}\right)=r\cos(\theta-\phi)$,
one has
\begin{align}
0=&\int_{\gamma}\int_{M}(\frac{-2n+r^2}{r^4}+\frac{8r^2\cos^2(\theta-\phi)}{r^6})
e^{-\frac{r^2}{4}}r^{n-1}d\mu_Mds \notag \\
=&\int_{\gamma}(\frac{-2n+r^2}{r^4}+\frac{8r^2\cos^2(\theta-\phi)}{r^6})
e^{-\frac{r^2}{4}}r^{n-1}ds\int_{M}d\mu_M. \notag
\end{align}
Then it get the equation (\ref{p502}) immediately since
$\int_Md\mu_M\neq0$.
\end{proof}
Next, we want to find a vector field $w_0$ in $\Gamma(T\psi(M))$
with nice special properties that will be needed in proving Theorem
\ref{t5} and Theorem \ref{thm6}.
\begin{lem}\label{l5}
Let $\psi :M^{n-1}\rightarrow \mathbb S^{2n-1}\subset\mathbb{C}^n$
be a minimal Legendrian immersion. Then there exists a nonzero
vector field $w_0$ in $\Gamma(T\psi(M))$ satisfying
\begin{align}
\frac{\int_M|\nabla^{M}w_0|^{2}-|\langle A^{M,\mathbb S},Jw_0\rangle|^{2}d\mu }
{\int_M|w_0|^{2}d\mu
}\leq 1 \quad\mbox{and} \quad \langle\nabla^M_x w_0,y\rangle=\langle\nabla^M_yw_0,x\rangle
\label{l501}
\end{align}
for any $x,y\in T\psi(M)$.
\end{lem}
\begin{rem}
The condition $\langle\nabla^M_x
w_0,y\rangle=\langle\nabla^M_yw_0,x\rangle$ implies that
$\frac{1}{r^2}J(\gamma w_0)$ induces a Lagrangian variation.
\end{rem}
\begin{proof}
Define
\begin{align}
f(y)=\int_M|\nabla^My|^2-|\langle A^{M,\mathbb S},Jy\rangle|^2d\mu \notag
\end{align}
for $y\in \Gamma (T\psi(M))$. Let $E_1,...,E_{2n}$ be the standard
basis for $\mathbb C^n$ with $E_{\alpha+n}=JE_{\alpha}$ for
$\alpha=1,...,n$. We claim that there exists a $\beta_0$ in
$\{1,...,2n\}$ such that $w_0=E_{\beta_0}^{\top}$ is a nonzero
vector field satisfying $f(w_0)\leq\int_M|w_0|^2d\mu$, where
$E_{\beta_0}^{\top}$ is the projection of $E_{\beta_0}$ into the
tangent space of $\psi(M)$. For fixed $q\in \psi (M)$, choose a
local normal coordinate system $x^1,...,x^{n-1}$ at $q$. Denote
$\partial_j=\frac{\partial}{\partial x^j}$. We have
\begin{align}
\langle \nabla^M_{\partial_k}(E_{\beta}^{\top}), \partial_j\rangle
=\langle \frac{\partial}{\partial x^k}(E_{\beta}-E_{\beta}^{\perp}), \partial_j\rangle
=-\langle \frac{\partial}{\partial x^k}E_{\beta}^{\perp}, \partial_j\rangle
=\langle E_{\beta}, A^M_{jk}\rangle, \label{l51}
\end{align}
where $E_{\beta}^{\perp}$ is the normal part of $E_{\beta}$. Since
the map $\psi$ is a Legendrian immersion into $\mathbb S^{2n-1}$,
the span $\{\partial_1,...,\partial_{n-1},X^M\}$ is a Lagrangian
plane in $\mathbb C^n$. It gives
\begin{align}
&A_{kj}^{M}=A_{kj}^{M,\mathbb S}+\langle A_{kj}^M,X^M\rangle X^M=A_{kj}^{M,\mathbb S}
-\delta_{kj} X^M \quad\mbox{at } q\label{l52}
\end{align}
and the second fundamental form $A^{M,\mathbb S}_{jk}$ of the
submanifold $\psi(M)$ in $\mathbb S^{2n-1}$ is orthogonal $JX^M$
because that
\begin{align}\langle A^{M,\mathbb
S}_{kj},JX^M\rangle =\langle\frac{\partial}{\partial
x^k}(\partial_j),JX^M\rangle
=-\langle\partial_j,J\partial_k\rangle=0. \notag
\end{align}
Since $\partial_l$ and $X^M$ are orthogonal, we have
$({JA^{M,\mathbb S}})^{\top}=JA^{M,\mathbb S}$. Recall that $\psi$
is a minimal immersion in $\mathbb S^{2n-1}$ and hence $H^{M,\mathbb
S}=0$. Combining the equations (\ref{l51}) and (\ref{l52}), the
first term of $f(E^{\top}_{\beta})$ can be simplified as
\begin{align}
|\nabla ^{M}(E_{\beta}^{\top})|^{2}
&=\underset{j,k =1}{\overset{n-1}{\sum }}|\langle E_{\beta},A_{kj}^{M,\mathbb S}\rangle
-\langle E_{\beta},\delta_{kj} X^M\rangle|^2 \notag \\
&=|\langle E_{\beta},A^{M,\mathbb S}\rangle|^{2}
-2\langle E_{\beta},H^{M,\mathbb S}\rangle\langle E_{\beta},X^M\rangle
+(n-1)\langle E_{\beta},X^M\rangle^{2} \notag \\
&=|\langle E_{\beta},A^{M,\mathbb S}\rangle|^{2}
+(n-1)\langle E_{\beta},X^M\rangle^{2} \quad \mbox{at} \quad q.
\label{l53}
\end{align}
Using the equality $({JA^{M,\mathbb S}})^{\top}=JA^{M,\mathbb S}$,
the second term of $f(E^{\top}_{\beta})$ can be simplified as
\begin{align}
\langle A^{M,\mathbb S},J(E_{\beta}^{\top})\rangle =-\langle
JA^{M,\mathbb S},E_{\beta}^{\top}\rangle =-\langle JA^{M,\mathbb
S},E_{\beta}\rangle =\langle A^{M,\mathbb S},JE_{\beta}\rangle.
\label{l54}
\end{align}
Combining (\ref{l53}) and (\ref{l54}), it gives
\begin{align}
f(E_{\alpha}^{\top})&=\int_M\left(|\langle E_{\alpha},A^{M,\mathbb S}\rangle|^{2}
+(n-1)\langle E_{\alpha},X^M\rangle^{2}-|\langle E_{\alpha+n},A^{M,\mathbb S}\rangle|^{2}
\right)d\mu, \label{l57}\\
f(E_{\alpha+n}^{\top})&=\int_M\left(|\langle E_{\alpha+n},A^{M,\mathbb S}\rangle|^{2}
+(n-1)\langle E_{\alpha+n},X^M\rangle^{2}-|\langle E_{\alpha},A^{M,\mathbb S}\rangle|^{2}
\right)d\mu \label{l58}
\end{align}
for $\alpha=1,...,n$. Summing (\ref{l57}) and (\ref{l58}) over
$\alpha=1,...,n$  gives
\begin{align}
\overset{n}{\underset{\alpha=1}{\sum}}\left(f(E^{\top}_{\alpha})+f(E^{\top}_{\alpha+n})\right)
=(n-1)\overset{2n}{\underset{\beta=1}{\sum}}\int_M\langle E_{\beta},X^M\rangle^{2}d\mu
=(n-1)\int_{M}d\mu \label{l59}
\end{align}
since $|X^M|=1$.

On the other hand, we have $\underset{\beta =1}{\overset{2n}{\sum
}}|E^{\top}_{\beta}|^{2} =\underset{\beta
=1}{\overset{2n}{\sum}}\overset{n-1}{\underset{j=1}{\sum }} \langle
E_{\beta},\partial_j\rangle^2=\overset{n-1}{\underset{j=1}{\sum
}}|\partial_j|^2=n-1$ at $q$ because $\partial_1,...,\partial_{n-1}$
is an orthonormal basis for $T_q\psi(M)$. Plugging it into
(\ref{l59}), we get
\begin{align}
\underset{\beta =1}{\overset{2n}{\sum }}\int_
M\left(|\nabla ^{M}(E^{\top}_{\beta})|^{2}-|\langle A^{M,\mathbb S},J(E^{\top}_{\beta
})\rangle|^{2}\right)d\mu =\underset{\beta =1}{\overset{2n}{\sum }}\int_
M|E^{\top}_{\beta}|^{2}d\mu. \notag
\end{align}
Therefore, there exists a $\beta_0$ in $\{1,..,2n\}$ such that
$E^{\top}_{\beta_0}$ is a nonzero vector field and
\begin{align}
\int_M\left(|\nabla ^{M}(E^{\top}_{\beta_0})|^{2}
-|\langle A^{M,\mathbb S},J(E^{\top}_{\beta_0})\rangle|^{2}\right)d\mu\leq
\int_M|E^{\top}_{\beta_0}|^{2}d\mu. \notag
\end{align}
Which is the inequality in (\ref{l501}). Using (\ref{l51}), $\langle
E_{\beta_0},A^M_{jk}\rangle$ is symmetric for $j,k$, it follows that
the vector field $w_0=E_{\beta_0}^{\top}$ satisfies both conditions
in (\ref{l501}).
\end{proof}
Now we are ready to proved Theorem \ref{t5}:

\begin{proof}[Proof of Theorem \ref{t5}] 

By Theorem \ref{thm3}, it suffices to construct a smooth normal
vector field $V$ such that (\ref{t301}) holds while
$\int_{\Sigma}\langle V,-L^{\perp}V\rangle
e^{-\frac{|X|^2}{4}}d\mu<0$. Assume $V=J(\gamma w)$, where $w\in
\Gamma(T\psi(M))$ would be chosen later. Because $H$ is parallel
to $Ju_s$ (see 
\cite{An}, p.40), we have $\int_{\Sigma}\langle V,H\rangle
e^{-\frac{|X|^2}{4}}d\mu=0$ and the first condition in (\ref{t301})
is satisfied. The second integral in (\ref{t301}) is
\begin{align}
\int_{\Sigma}Ve^{-\frac{|X|^{2}}{4}}d\mu
=i\int_{\gamma }\gamma e^{-\frac{r^{2}}{4}}r^{n-1}ds
\int_{M}wd\mu _{M}. \notag
\end{align}
Recall that the construction of $\gamma$ in \cite{An} is made by
$m>1$ pieces $\Gamma_1,...,\Gamma_m$ which each corresponds one
period of curvature function. (In particular, when $\gamma$ is the
circle $\mathbb S^1({\sqrt{2n}})$, we take $m=2$.) Every piece
$\Gamma_i$ is the same as $\Gamma_1$ up to a rotation. Suppose the
rotation index of $\gamma$ is $l$. Then we have
\begin{align}
\int_{\gamma }\gamma e^{-\frac{r^{2}}{4}}r^{n-1}ds
&=\underset{j =1}{\overset{m}{\sum }}\int_{\Gamma_j }e^{-\frac{r^{2}}{4}}r^ne^{i\phi}ds
 \notag \\
&=\int_{\Gamma_1 }e^{-\frac{r^{2}}{4}}r^ne^{i\phi}(1+e^{i\frac{2l\pi}{m}}+...+e^{i
\frac{(m-1)l}{m}\cdot 2\pi})ds
=0, \notag
\end{align}
since $1+e^{i\frac{2\pi}{m}}+...+e^{i \frac{(m-1)l}{m}\cdot 2\pi}=0$
for $m>1$. Therefore, the second integral condition in (\ref{t301})
also holds.

For the case $n\geq 3$, we choose $w=w_0$ satisfying (\ref{l501})
and $V_0=J(\gamma w_0)$. Plugging the first inequality of
(\ref{l501}) into (\ref{l303}), the weighted inner product $\langle
V_0,-L^{\perp}V_0 \rangle _e$ becomes
\begin{align}
&\int_{\Sigma }\langle V_0,-L^{\perp }V_0\rangle e^{-\frac{|X|^{2}}{4}}d\mu \notag \\
\leq&-\int_{\gamma}\left(\frac{1}{2}r^{2}-3+4\sin ^{2}
 (\theta -\phi)\right)e^{\frac{-r^2}{4}}r^{n-1}ds\int_M|w_0|^{2}d\mu_M \notag \\
= &-\int_{\gamma }\Bigl(\bigl(n-3+4\sin^{2}(\theta -\phi )\bigr)\Bigr)
e^{-\frac{r^{2}}{4}}r^{n-1}ds\int_{M}|w_0|^{2}d\mu _{M} \notag \\
<&0. \notag
\end{align}
We use  (\ref{p501}) to conclude the equality above.

For the case $n=2$, the only minimal Legendrian curves in $\mathbb
S^3$ are great circles. They are totally geodesic in $\mathbb S^3$.
Therefore, the weighted inner product $\langle
V,-L^{\perp}V\rangle_e$ can be simplified as
\begin{align}
&\int_{\Sigma }\langle V,-L^{\perp }V\rangle e^{-\frac{|X|^{2}}{4}}d\mu \notag \\
=&\int_{\gamma}e^{-\frac{r^{2}}{4}}r\Bigl(\int_{\mathbb S^1}|\nabla^{\mathbb S^1}w|^{2}
-\left(\frac{1}{2}r^2-2+4\sin
^{2}(\theta -\phi)\right)|w|^{2}d\mu _{\mathbb S^1}\Bigr)ds. \notag \\
=&\int_{\gamma}e^{-\frac{r^{2}}{4}}r\Bigl(\int_{\mathbb S^1}|\nabla^{\mathbb S^1}w|^{2}
-4\sin^{2}(\theta -\phi)|w|^{2}d\mu _{\mathbb S^1}\Bigr)ds. \notag
\end{align}
Here we use (\ref{p501}) again to get the last equality. Finally, by
choosing $w$ to be the tangent vector of the great circle, which is
a parallel vector field, we can make the weighted inner product
negative.
\end{proof}
\subsection{The unstability for Lagrangian variations}
Since Anciaux's examples are Lagrangian, it is natural to
investigate whether these examples are still unstable under the more
restricted Lagrangian variations. That is, for variations  from
the deformation of Lagrangian submanifolds. A simple calculation
shows that a vector field $V$ induces a Lagrangian variation if
and only if the associated one form  $\alpha_V=\omega(V,\cdot)$ is
closed, i.e.
\begin{align}
\langle \nabla^{\perp}_XV,JY\rangle=\langle \nabla^{\perp}_YV,JX\rangle, \label{t6}
\end{align}
where $\nabla^{\perp}$ is the normal connection on $N\Sigma$ and
$X,~Y\in T\Sigma$. For the problem, we can prove
\begin{thm}\label{thm6}
Let $\Sigma$ be an n-dimensional closed Lagrangian self-shrinker as
in Lemma \ref{l2}. Then $\Sigma$ is $F$-unstable under Lagrangian
variations for the following cases
\begin{itemize}
\item[\rm{(i)}] $n=2$ or  $n\geq7$,
\item[\rm{(ii)}] $2<n<7$, and
$E\in[\frac{1}{\sqrt{2}}E_{max},E_{max}]$,
\end{itemize}
where $E$ and $E_{max}$ are described in ($\ref{l202}$).
\end{thm}

 Because
$\langle\nabla^{\perp}_{u_s}V,Ju_j\rangle\neq\langle\nabla^{\perp}_{u_j}V,Ju_s\rangle$
for $V\in N_0$, it  does not induce a Lagrangian variation.
Thus to prove the theorem, we need to consider variations different from those
in \S \ref{us42}. We now define a new set $N_1$ as follows:
\begin{align}
N_1=\{V|&V=\frac{1}{r^2}J(\gamma w), ~\mbox{where}~w\in \Gamma(T\psi(M))
~\mbox{satisfies}~ \notag\\
&\langle\nabla^M_{x}w,y\rangle
=\langle\nabla^M_{y}w,x\rangle, \mbox{ for all }x,y\in T\psi(M)\}. \notag
\end{align}
For $V\in N_1$, we claim that $V$  satisfies the equation (\ref{t6}) and
hence indeed induces a Lagrangian variation. Suppose
$V=\frac{1}{r^2}J(\gamma w)$. Noting
that $\gamma'=e^{i\theta}$, $\langle V,Ju_s\rangle=0$, and $r'$ satisfying
(\ref{l201}), we therefore have
\begin{align}
\langle\nabla^{\perp}_{u_s}V,Ju_j\rangle
&=-\frac{2r'}{r^3}\langle J(\gamma w),J(\gamma e_j)\rangle
+\frac{1}{r^2}\langle J(\gamma'w),J(\gamma e_j)\rangle \notag \\
&=-\frac{\cos(\theta-\phi)}{r}\langle w,e_j\rangle, \notag \\
\langle\nabla^{\perp}_{u_j}V,Ju_s\rangle
&=-\langle V, \nabla^{\perp}_{u_j}Ju_s\rangle
=-\frac{1}{r^2}\langle J(\gamma w),J(\gamma' e_j)\rangle
=-\frac{\cos(\theta-\phi)}{r}\langle w,e_j\rangle, \notag \\
\langle\nabla^{\perp}_{u_k}V,Ju_j\rangle
&=\frac{1}{r^2}\langle\frac{\partial}{\partial x_k}J(\gamma w),J(\gamma e_j)\rangle
=\langle\nabla^M_{e_k}w,e_j\rangle \notag \\
&=\langle\nabla^M_{e_j}w,e_k\rangle
=\langle\nabla^{\perp}_{u_j}V,Ju_k\rangle. \notag
\end{align}
This proves the claim.

For $V\in N_1$, the operator $\langle V,-L^{\perp}V\rangle_e$ can be
simplified as in the following lemma.
\begin{lem} \label{l4}
Assume that $\Sigma$ is a closed Lagrangian self-shrinker as in
Lemma~\ref{l2} and $V\in N_1$ is represented by
$\frac{1}{r^2}J(\gamma w)$. The second fundamental forms of $\Sigma
$\ in $\mathbb C^{n}$ and of $\psi(M)$ in $\mathbb S^{2n-1}$ are
denoted by $A^{\Sigma}$ and $A^{M,\mathbb{S}}$, respectively. Then we
have
\begin{align}
(i) ~~~ &|\langle A^{\Sigma},V\rangle|^{2}=\frac{1}{r^4}|\langle A^{M,\mathbb S},Jw\rangle|^{2}
+\frac{2}{r^4}\sin^{2}(\theta -\phi )|w|^{2},\label{l401} \\
(ii) ~~&|\nabla ^{\perp }V|^{2}=\frac{1}{r^4}|\nabla ^{M}w|^{2}
+\frac{2\cos ^{2}(\theta -\phi)}{r^4}|w|^{2},\label{l402} \\
(iii) ~&\langle V,-L^{\perp}V\rangle_e
=-\int_{\gamma}\left(\frac{1}{2}r^{2}-2+4\sin ^{2}
 (\theta -\phi)\right)e^{\frac{-r^2}{4}}r^{n-5}ds\int_M|w|^{2}d\mu_M \notag \\
&~~~~~~~~~~~~~~~+\int_{\gamma}e^{\frac{-r^2}{4}}r^{n-5}ds
\int_M\left(|\nabla ^{M}w|^{2}-|\langle A^{M,\mathbb S},Jw\rangle|^{2}\right)d\mu_M.\label{l403}
 \end{align}
\end{lem}
\begin{proof}
(i) For $V\in N_1$, there exist $V_0\in N_0$ such that
$V=\frac{1}{r^2}V_0=\frac{1}{r^2}J(\gamma w)$. Using
the equation (\ref{l301}), we have
\begin{align}
|\langle A^{\Sigma},V\rangle|^{2}=\frac{1}{r^4}|\langle A^{\Sigma},V_0\rangle|^{2}
=\frac{1}{r^4}|\langle A^{M,\mathbb S},Jw\rangle|^{2}
+\frac{2}{r^4}\sin^{2}(\theta -\phi )|w|^{2}. \notag
\end{align}
(ii) Using the equations (\ref{t6}) and (\ref{l34}), we have
\begin{align}
\hspace{-0.35cm}
&\langle \nabla^{\perp} _{u_k}\frac{1}{r^2}J(\gamma w),Ju_{j}\rangle
=\frac{1}{r^2}\langle \nabla^{\perp} _{u_k}J(\gamma w),Ju_{j}\rangle
=\langle \nabla _{e_k}^{M}w,e_{j}\rangle \notag \\
&\langle\nabla^{\perp}_{u_k}\frac{1}{r^2}J(\gamma w),Ju_{s}\rangle
=\frac{1}{r^2}\langle\nabla^{\perp}_{u_k}J(\gamma w),Ju_{s}\rangle
=-\frac{1}{r}\cos (\theta -\phi)\langle w,e_j\rangle \label{l41} \\
&\langle\nabla^{\perp}_{u_s}\frac{1}{r^2}J(\gamma w),Ju_s\rangle
=\frac{-2r'}{r^3}\langle J(\gamma w),Ju_s\rangle
+\frac{1}{r^2}\langle \nabla^{\perp}_{u_s}J(\gamma w),Ju_s\rangle
=0. \notag
\end{align}
Using (\ref{l41}) and (\ref{t6}), the computation at $p$  gives
\begin{align}
&|\nabla ^{\perp }V|^{2}
=\langle\nabla^{\perp} _{u_{\alpha}}\frac{1}{r^2}J(\gamma w),
\nabla^{\perp} _{u_{\beta}}\frac{1}{r^2}J(\gamma w)\rangle
g^{\alpha \beta } \notag \\
=&\overset{n-1}{\underset{k=1}{\sum}}\langle\nabla^{\perp}_{u_k}\frac{1}{r^2}J(\gamma w),
\nabla^{\perp}_{u_k}\frac{1}{r^2}J(\gamma
w)\rangle\frac{1}{r^{2}}+\langle\nabla _{u_s}\frac{1}{r^2}J(\gamma w),
\nabla _{u_s}\frac{1}{r^2}J(\gamma w)\rangle \notag \\
=&\left(\overset{n-1}{\underset{j,k=1}{\sum}}\langle\nabla^{\perp}
_{u_k}\frac{1}{r^2}J(\gamma w),\frac{Ju_{j}}{r}\rangle^2
+\overset{n-1}{\underset{k=1}{\sum}}\langle\nabla^{\perp}_{u_k}\frac{1}{r^2}J(\gamma
w),Ju_s\rangle^2\right)\frac{1}{r^{2}} \notag \\
&+\overset{n-1}{\underset{j=1}{\sum}}\langle\nabla^{\perp}_{u_s}
\frac{1}{r^2}J(\gamma w),\frac{Ju_j}{r}\rangle^2 \notag \\
=&\frac{1}{r^4}\overset{n-1}{\underset{j,k=1}{\sum}}\langle\nabla_{e_k}^M w,e_j\rangle^2
+\frac{2}{r^4}\overset{n-1}{\underset{j=1}{\sum}}\cos^2(\theta-\phi)\langle w,e_j\rangle^2 \notag \\
=&\frac{1}{r^4}|\nabla ^{M}w|^{2}+\frac{2\cos ^{2}(\theta -\phi)}{r^4}|w|^{2}. \notag
\end{align}
(iii)
Plugging (\ref{l401}) and (\ref{l402}) into (\ref{r2}), and
using
$e^\frac{-|X|^2}{4}d\mu_{\Sigma}=e^{-\frac{r^2}{4}}r^{n-1}dsd\mu_M$,
we get
\begin{align}
&\langle V, -L^{\perp}V\rangle_e \notag \\
=&\int_{\Sigma}\left(|\nabla^{\perp}V|^2-|\langle A^{\Sigma}, V\rangle|^2-\frac{1}{2}|V|^2\right)
e^{-\frac{|X|^2}{4}}d\mu_{\Sigma} \notag \\
=&\int_{\gamma}\int_M\Big(|\nabla ^{M}w|^{2}+2\cos ^{2}(\theta
-\phi )|w|^{2}-\left(|\langle A^{M,\mathbb S},Jw\rangle|^{2}
+2\sin^{2}(\theta -\phi )|w|^{2}\right) \notag \\
&-\frac{1}{2}r^2|w|^2\Big) e^{-\frac{r^2}{4}}r^{n-5}d\mu_mds \notag \\
=&-\int_{\gamma}\left(\frac{1}{2}r^{2}-2+4\sin ^{2}
 (\theta -\phi)\right)e^{\frac{-r^2}{4}}r^{n-5}ds\int_M|w|^{2}d\mu_M \notag \\
&+\int_{\gamma}e^{\frac{-r^2}{4}}r^{n-5}ds
\int_M\left(|\nabla ^{M}w|^{2}-|\langle A^{M,\mathbb S},Jw\rangle|^{2}\right)d\mu_M.\notag
\end{align}
Thus (iii) is proved.
\end{proof}

\begin{proof}[Proof of Theorem \ref{thm6}]

By Theorem \ref{thm3}, it suffices to construct a smooth normal
Lagrangian variation $V$ such that $\int_{\Sigma}\langle
V,-L^{\perp}V\rangle e^{-\frac{|X|^2}{4}}d\mu<0$ while  (\ref{t301})
holds. Assume $V=\frac{1}{r^2}J(\gamma w)\in N_1$, where $w\in
\Gamma(T\psi(M))$ will be chosen later. Similar to the proof of
Theorem \ref{t5}, both the conditions in (\ref{t301}) hold.

We now  further specify $V$, so that $\int_{\Sigma}\langle V,-L^{\perp}V\rangle
e^{-\frac{|X|^2}{4}}d\mu<0$. When $n\geq 3$, we choose $w=w_0$ satisfying (\ref{l501}).
Then
$V_1=\frac{1}{r^2}J(\gamma w_0)$ is in $N_1$. From (\ref{l501}) and
(\ref{l403}), the weighted inner product $\langle V_1,-L^{\perp}V_1
\rangle _e$ becomes
\begin{align}
&\int_{\Sigma }\langle V_1,-L^{\perp }V_1\rangle e^{-\frac{|X|^{2}}{4}}d\mu \notag \\
\leq&-\int_{\gamma}\left(\frac{1}{2}r^{2}-3+4\sin ^{2}
 (\theta -\phi)\right)e^{\frac{-r^2}{4}}r^{n-5}ds\int_M|w_0|^{2}d\mu_M \notag \\
= &-\int_{\gamma }\Bigl(\bigl(n-3+4\sin^{2}(\theta -\phi)-4\cos^{2}(\theta -\phi )\bigr)\Bigr)
e^{-\frac{r^{2}}{4}}r^{n-5}ds\int_{M}|w_0|^{2}d\mu _{M}, \notag
\end{align}
 where
(\ref{p502}) is used to conclude the above equality. Thus
it suffices to show that
$f(s)=n-3+4\sin^2(\theta-\phi)-4\cos^{2}(\theta -\phi
)$ is nonnegative and positive at some point. Because $|\cos(\theta -\phi
)|\leq 1$, the
function $f$ is clearly nonnegative and positive somewhere for $n\geq7$.
 When $E\in[\frac{1}{\sqrt{2}}E_{max},E_{max}]$,
 one has $\sin(\theta-\phi)\in[\frac{1}{\sqrt{2}},1]$ from (\ref{l202}) and
 hence
$f(s)$ is nonnegative and positive somewhere.

In the case  $n=2$,  the only minimal Legendrian
curves in $\mathbb S^3$ are great circles which are totally geodesic.
 Choosing $w_1$ to be the tangent vector of the
great circle, we have $|\nabla^{\mathbb S^1}w_1|=0$ and $|w_1|=1$.
The vector field $V_1=\frac{1}{r^2}J(\gamma w_1)$ gives a Lagrangian
variation and the weighted inner product $\langle
V_1,-L^{\perp}V_1\rangle_e$  in (\ref{l403}) can be simplified as
\begin{align}
&\int_{\Sigma }\langle V_1,-L^{\perp }V_1\rangle e^{-\frac{|X|^{2}}{4}}d\mu \notag \\
=&-\int_{\gamma}\left(\frac{1}{2}r^2-2+4\sin^2(\theta-\phi)\right)
e^{-\frac{r^{2}}{4}}r^{-3}ds\int_{\mathbb S^1}|w|^2d\mu_{\mathbb S^1} \notag \\
=&-2\pi\int_{\gamma}\left(\frac{1}{2}r^2
+2\left(\sin^2(\theta-\phi)-\cos^2(\theta-\phi)\right)\right)
e^{-\frac{r^{2}}{4}}r^{-3}ds\notag
\end{align}
Using (\ref{p502}), it follows that
\begin{align}
\int_{\gamma}\frac{1}{2}r^2e^{-\frac{r^{2}}{4}}r^{-3}ds
=\int_{\gamma}2\left(\sin^2(\theta-\phi)-\cos^2(\theta-\phi)\right)
e^{-\frac{r^{2}}{4}}r^{-3}ds. \notag
\end{align}
Therefore, $\langle V_1,-L^{\perp}V_1\rangle_e
=-2\pi\int_{\gamma}r^2e^{-\frac{r^{2}}{4}}r^{-3}ds<0$,
and concludes the Lagrangian unstability in Theorem \ref{thm6}.
\end{proof}


\begin{thebibliography}{99}
\bibitem{AL} U. Abresch and J. Langer, The normalized curve shortening flow and homothetic solutions.
{\it J. Differential Geom.} 23 (1986), no. 2, 175-196.
\bibitem{An} H. Anciaux, Construction of Lagrangian self-similar solutions to the mean
curvature flow in $\mathbb C^n$. {\it Geom. Dedicata} 120 (2006),
37-48.
\bibitem{A} S. Angenent, Shrinking doughnuts, {\it Nonlinear diffusion equations and their equilibrium states,}
Birkha\"user, Boston-Basel-Berlin, 3, 21-38, 1992
\bibitem{AChI} S. B. Angenent, D. L. Chopp, and T. Ilmanen. A computed example of nonuniqueness of mean
curvature flow in $\mathbb R^3$. {\it Comm. Partial Differential
Equations}, 20 (1995), no. 11-12, 1937-1958.
\bibitem{CM} T.H. Colding and W.P. Minicozzi, Generic mean curvature flow I; generic
singularities, {\it to appear in Ann. Math.} 
\bibitem{JLT} Dominic Joyce; Yng-Ing Lee; Mao-Pei Tsui, Self-similar solutions and translating
solitions for Lagrangian mean curvature flow. {\it J Differential
Geom.} 84 (2010), no. 1, 127-161.
\bibitem{Hu1} G. Huisken, Flow by mean
curvature of convex surfaces into spheres. {\it J. Differential Geom.}
{\bf 20} (1984), no. 1, 237--266.
\bibitem{Hu2} G. Huisken,
Asymptotic behavior for singulairites of the mean curvature flow.
{\it J. Differential Geom.} {\bf 31} (1990), no. 1, 285--299.
\bibitem{Hu3} G. Huisken,
Local and global behaviour of hypersurfaces moving by mean
curvature. Differential geometry: {\it partial differential
equations on manifolds} (Los Angeles, CA, 1990), {\it Proc. Sympos.
Pure Math.}, {\bf 54}, Part 1, Amer. Math. Soc., Providence, RI,
(1993), 175-191.
\bibitem{I1} T. Ilmanen, {\it Singularities of mean curvature flow of surfaces},
preprint, 1995, http://www.math.ethz.ch/~ilmanen/papers/pub.html.
\bibitem{LW2} Y.-I. Lee and M.-T. Wang, {\it Hamiltonian
stationary cones and self-similar solutions in higher dimension},
Trans. Amer. Math. Soc. {\bf 362} (2010), 1491--1503.
\bibitem{Sm} K. Smoczyk, Self-shrinkers of the mean curvature flow in arbitrary codimension,
{\it International Mathematics Research Notices}, {\bf 48} (2005),
2983-3004.
\bibitem{St} A. Stone, A density function and the structure of singularities of the mean
curvature flow. Calc. Var. {\it Partial Differential Equations 2}
(1994), no. 4, 443-480.
\bibitem{Wh} B. White,
A local regularity theorem for classical mean curvature flow.
{\it Ann. of Math.} (2) {\bf 161} (2005), no. 3, 1487-1519.
\end{thebibliography}
\end{document}